\documentclass[11pt, oneside]{amsart}
 \usepackage[text={5.58in,8.5in},centering,letterpaper,dvips]{geometry}
\usepackage{graphicx}
\usepackage{amsfonts}
\usepackage[dvipsnames]{xcolor}
\usepackage{subcaption}
\usepackage{epsf}
\usepackage{amssymb}
\usepackage{amsmath}
\usepackage{amscd}
\usepackage{pdfpages}
\usepackage{fancyhdr}
\usepackage{setspace}
\usepackage{soul} 
\usepackage[all]{xy}
\usepackage{verbatim}
\usepackage{enumerate}
\usepackage[colorlinks=true, urlcolor=NavyBlue, linkcolor=NavyBlue, citecolor=NavyBlue]{hyperref}

\theoremstyle{theorem}
\newtheorem{theorem}{Theorem}[section]
\newtheorem{proposition}[theorem]{Proposition}
\newtheorem{lemma}[theorem]{Lemma}
\newtheorem{question}[theorem]{Question}

\newtheorem{conjecture}[theorem]{Conjecture}

\makeatletter
\newtheorem*{rep@theorem}{\rep@title}
\newcommand{\newreptheorem}[2]{%
\newenvironment{rep#1}[1]{%
 \def\rep@title{#2 \ref{##1}}%
 \begin{rep@theorem}}%
 {\end{rep@theorem}}}
\makeatother

\newreptheorem{theorem}{Theorem}
\newreptheorem{lemma}{Lemma}
\newreptheorem{question}{Question}
\newreptheorem{corollary}{Corollary}
\newreptheorem{proposition}{Proposition}

\theoremstyle{definition}
\newtheorem{definition}[theorem]{Definition}
\newtheorem{remark}[theorem]{Remark}

\newtheorem{example}[theorem]{Example}

\newcommand{\F}{\mathcal{F}}
\newcommand{\Z}{\mathbb{Z}}

\newcommand{\N}{\mathbb{N}}

\newcommand{\Dd}{\mathcal D}

\newcommand{\Rr}{\mathfrak R}

\newcommand{\A}{\alpha}

\newcommand{\pd}{\partial}

\newcommand{\K}{\mathcal K}
\newcommand{\D}{\mathcal D}

\newcommand{\orange}[1]{\textcolor{orange}{#1}}

\newcommand{\red}{\textcolor{red}}
\newcommand{\blue}{\textcolor{blue}}

\makeatletter
\def\@seccntformat#1{%
  \protect\textup{\protect\@secnumfont
    \ifnum\pdfstrcmp{subsection}{#1}=0 \bfseries\fi
    \csname the#1\endcsname
    \protect\@secnumpunct
  }%
}  
\makeatother

\begin{document}

\rhead{\thepage}
\lhead{\author}
\thispagestyle{empty}


\raggedbottom
\pagenumbering{arabic}
\setcounter{section}{0}


\title{Ribbon numbers of 12-crossing knots}

\author{Xianhao An}
\address{Department of Mathematics, University of Illinois Urbana-Champaign
Champaign, IL 61820}
\email{xianhao2@illinois.edu}

\author{Matthew Aronin}
\address{School of Mathematics, Georgia Institute of Technology, Atlanta, GA 30332}
\email{maronin@gatech.edu}

\author{David Cates}
\address{Department of Mathematics, Texas A\&M University, College Station, TX 77843}
\email{dkrc@tamu.edu}

\author{Ansel Goh}
\address{Department of Mathematics, The University of Texas at Austin, Austin, TX 78712}
\email{anselgoh@utexas.edu}
\urladdr{https://sites.google.com/view/anselgoh}

\author{Benjamin Kirn}
\address{Department of Physics, University of Maryland, College Park, MD 20742}

\author{Josh Krienke}
\address{Bard College
Annandale-on-Hudson, NY 12504}
\email{joshjoekrienke@gmail.com}
\urladdr{https://www.joshkrienke.com}

\author{Minyi Liang}
\address{Department of Mathematics, Jilin University, Changchun, Jilin, China}

\author{Samuel Lowery}
\address{Department of Mathematics, The Ohio State University, Columbus, OH 43210}
\email{lowery.257@osu.edu}

\author{Ege Malkoc}
\address{Department of Mathematics, Istanbul Bilgi University, Istanbul, T\"urkiye}
\email{malkocege215@gmail.com}

\author{Jeffrey Meier}
\address{Department of Mathematics, Western Washington University, Bellingham, WA 98229}
\email{jeffrey.meier@wwu.edu}
\urladdr{http://jeffreymeier.org} 

\author{Max Natonson}
\address{Department of Mathematics, University of Michigan, Ann Arbor, MI 48104}
\email{natonson@umich.edu}

\author{Veljko Radi\'c}
\address{Faculty of Mathematics, University of Belgrade, Belgrade, Serbia}
\email{vradicc@gmail.com}

\author{Yavuz Rodoplu}
\address{Department of Mathematics, University of Nebraska Lincoln, Lincoln, NE 68508}
\email{yrodoplu2@unl.edu}

\author{Bhaswati Saha}
\address{Department of Mathematics, Presidency University, Kolkata, India.}

\author{Evan Scott}
\address{Department of Mathematics, CUNY Graduate Center New York, NY 10016}
\email{escott@gradcenter.cuny.edu}

\author{Roman Simkins}
\address{Department of Mathematics, Ohio University, Athens, OH 45701}
\email{roman.simkins@gmail.com}
\urladdr{https://romansimkins.com/}

\author{Alexander Zupan}
\address{Department of Mathematics, University of Nebraska--Lincoln, Lincoln, NE 68588}
\email{zupan@unl.edu}
\urladdr{https://math.unl.edu/azupan2}


\begin{abstract}
	The ribbon number of a knot is the minimum number of ribbon singularities among all ribbon disks bounded by that knot.  In this paper, we build on the systematic treatment of this knot invariant initiated in recent work of Friedl, Misev, and Zupan.  We show that the set of Alexander polynomials of knots with ribbon number at most four contains 56 polynomials, and we use this set to compute the ribbon numbers for many 12-crossing knots.
	We also study higher-genus ribbon numbers of knots, presenting some examples that exhibit interesting behavior and establishing that the success of the Alexander polynomial at controlling genus-0 ribbon numbers does not extend to higher genera.
\end{abstract}

\maketitle

\section{Introduction}

The slice-ribbon conjecture of Fox, which asks whether every slice knot is ribbon, is one of the most famous open problems in knot theory.  A knot $K \subset S^3$ is smoothly \emph{slice} if $K$ bounds a smooth, properly embedded disk $\D$ in $B^4$, and $K$ is \emph{ribbon} if $K$ bounds such a disk $\D$ that has no maxima on its interior with respect to the radial Morse function on $B^4$.
In this case, $\D$ can be projected to an immersed disk in $S^3$ with only ribbon singularities; that is, a \emph{ribbon disk} for $K$.
The \emph{ribbon number} $r(K)$ of $K$ minimizes the number $r(\D)$ of ribbon singularities among all ribbon disks $\D$ bounded by $K$.

Recent work of Friedl, Misev, and Zupan shows how Alexander polynomials can be used as effective tool to catalogue the ribbon numbers of low-crossing knots.
They determine the ribbon number for all but three knots with eleven or fewer crossings~\cite{FMZ}.
In this paper, we extend that investigation to knots with 12 crossings.  Let $\Rr_r$ denote the set of all possible Alexander polynomials of ribbon knots $K$ such that $r(K) \leq r$.  In~\cite{FMZ}, the authors proved that for each $r$, the set $\Rr_r$ is finite and computable, and they determined that $\Rr_2$ and $\Rr_3$ have cardinalities three and ten, respectively.  We determine the cardinality of $\Rr_4$.

\begin{theorem}\label{thm:main}
	The set $\Rr_4$ has 56 elements, which are enumerated in Table~\ref{table:R4}.
\end{theorem}

The proof involves encoding the data necessary to compute the Alexander polynomial of $K$ in a marked graph called a \emph{ribbon code} and exhaustively enumerating all ribbon codes up to ribbon number four.
Using Theorem~\ref{thm:main} as an obstruction for a knot to have ribbon number four, we determine the ribbon number for many 12-crossing knots.

\begin{theorem}\label{thm:tab}
	The ribbon numbers for 54 of the 58 prime, non-alternating 12-crossing ribbon knots are given in Table~\ref{table:12n}.
	The ribbon numbers for 20 of the 49 prime, alternating 12-crossing ribbon knots are give in Table~\ref{table:12a}.
\end{theorem}

Our main tool for finding upper bounds on ribbon numbers are symmetric union presentations (Lemma~\ref{lem:symm}) and explicit constructions, either shown in Figures~\ref{fig:ribbonA} and~\ref{fig:ribbonB} or appearing elsewhere in the literature.
The lower bounds rely primarily on the sets $\Rr_2$, $\Rr_3$, and $\Rr_4$ (Proposition~\ref{prop:r2r3} and Theorem~\ref{thm:main}). 

In addition, we explore ribbon numbers with respect to higher genus ribbon surfaces, defining $r_g(K)$ to be the smallest number of ribbon intersections of any genus-$g$ ribbon surface $F$ bounded by $K$. 
With $g(K)$ denoting the Seifert genus of $K$, we note that for any $g \geq g(K)$, we have $r_g(K) = 0$, since $K$ bounds an embedded genus-$g$ surface.
The collection of all genus-$g$ ribbon numbers can be arranged as a tuple
\[ \mathfrak{r}(K) = (r_0(K),r_1(K),\dots,r_{g(K)}(K)),\]
which we refer to as the \emph{ribbon spectrum} of $K$.
(In a similar vein, the \emph{bridge spectrum} of a knot is defined and studied in~\cite{zupan-bridge}.)
The cellar door trick (resolving a single ribbon intersection; see Remark~\ref{rmk:spectrum}) yields the inequality $r_g(K) \leq r_{g-1}(K) - 1$, and we expect that generically, the ribbon spectrum takes a stair-step form:
$$\mathfrak{r}(K) = (r(K),r(K)-1,r(K)-2,\dots,0).$$
Indeed, this can be seen to be the case for any knot $K$ satisfying $r(K) = g(K)$, such as the generalized square knot $T_{p,q} \# \overline T_{p,q}$~\cite[Proposition~2.9]{FMZ}.
We prove there exist knots with multiple ``jumps" in their ribbon spectra:

\begin{theorem}
\label{thm:pretzel}
	For odd $q\in\Z$, the $5$--stranded pretzel knot $K = P(q,3,-3,3,-3)$ satisfies	
	$$\frak r(K) = (4,2,0).$$
\end{theorem}

Finally, a natural question arising from the work in~\cite{FMZ} is whether determinants or Alexander polynomials can be used to obstruct higher genus ribbon numbers.  We demonstrate that this obstruction only succeeds in the case of ribbon disks.
Note that if the degree of $\Delta_K(t)$ is $2d$, then $g(K) \geq d$, and by Proposition~\ref{prop:hgrn_bound} below,
\[ r_1(K) \geq g(K) - 1 \geq d - 1.\]
We prove that this is the \emph{only} restriction $\Delta_K(t)$ imposes on $r_1(K)$ among the collection of all ribbon knots.
\begin{theorem}
\label{thm:torus_alex}
	If $K$ is a ribbon knot with $\text{deg}(\Delta_K(t))=2d$, then there exists a ribbon knot $K'$ such that $\Delta_{K'} (t)=\Delta_K(t)$ and $r_1(K')=d-1$.
	In particular, for any $r$, the set of Alexander polynomials of ribbon knots $J$ with $r_1(J) \leq r$ is infinite, and the set of determinants of such $J$ is unbounded.
\end{theorem}

The final statement of the theorem contrasts Corollary 1.4 of~\cite{FMZ}, which asserts that if $K$ is a ribbon knot,
\[ \det(K) \leq (2^{r(K)} - 1)^2.\]
\begin{remark}
\label{rem:yasuda}
	We note that, as an invariant, ribbon number is closely related to the \emph{ribbon crossing number} $\text{r-cr}(\K)$ of a ribbon 2--knot $\K$.
	Any ribbon disk $\D$ can be perturbed to get an embedded disk in $B^4$ and then doubled to yield a ribbon 2--knot $\K(\D)$ with $r(\D)$ ribbon crossings.
	In addition, if $\Delta_{\K(\D)} = f(t)$, then $\Delta_{\pd D} = f(t) f(t^{-1})$.
	(This is discussed in detail in Section~3 of~\cite{FMZ}.)
	Thus, the same tools we use here can be used to understand ribbon crossing numbers, and indeed, similar investigations have been carried out in this other setting.
	
	In~\cite{yasuda:2-knot}, Yasuda enumerates all Alexander polynomials of ribbon 2-knots with ribbon crossing number at most four, and in theory, this list should agree with $\Rr_4$.
	However, our set contains all of polynomials in Yasuda's list and uncovers several missed cases: The last column of Table~\ref{table:R4} gives a 2-knot found by Yasuda with the given Alexander half-polynomial, so the six rows with no entry in this column correspond to 2-knots $\K$ with $\text{r-cr}(\K) \leq 4$ not appearing in~\cite{yasuda:2-knot}.
	(Note that we have included only one 2-knot from each mirror pair described by Yasuda for each polynomial.)
	Furthermore, the polynomial $\Delta_{K_{57}^2}(t)$ given in~\cite[Section~5.1]{yasuda:2-knot} contains a sign error on the highest-degree term; the correct polynomial is given in Table~\ref{table:R4}.
	Kanenobu gave a classification of Yasuda's ribbon 2-knots~\cite{kanen}.
\end{remark}

\begin{remark}
\label{rmk:howie}
	A combinatorial object equivalent to a ribbon code was used by Howie to investigate the (still open) question of whether the exterior of a ribbon disk (in $B^4$) is aspherical~\cite{howie_asphericity}.
	Howie called these objects \emph{weakly labeled oriented trees}, and, in the language of ribbon codes, he proved that the corresponding disk exteriors are aspherical provided any two vertices can be connected by a path containing at most three markings.
	This infinite class of ribbon codes includes all ribbon codes with ribbon number at most three and ribbon codes with ribbon number four and Structure (6), (7), or (8) in Figure~\ref{fig:R4_structures}.
\end{remark}

\subsection{Organization}

In Section~\ref{sec:prelim} we recall preliminary material, most of which comes from~\cite{FMZ}; however, we also introduce \emph{decomposability} for ribbon codes and a move called a \emph{leaf isotopy}.
In Section~\ref{sec:R4}, we characterize the possible structures for ribbon codes with ribbon number four, and we use this to give a complete list of such codes, which are presented in Tables~\ref{table:rc1-2},~\ref{table:rc3-5}, and~\ref{table:rc6-8}.
From this list, we deduce the members of $\Rr_4$.
In Section~\ref{sec:tab}, we use the set $\Rr_4$, together with techniques and results from Section~\ref{sec:prelim}, to determine the ribbon numbers for many knots with 12-crossings and to offer upper and lower bounds for others.  The results are presented in Tables~\ref{table:12n} and~\ref{table:12a}.
In Section~\ref{sec:hgrn}, we initiate the study of higher-genus ribbon numbers, paying particular attention to certain pretzel knots and establishing the ineffectiveness of the Alexander polynomial in this setting, and proving Theorems~\ref{thm:pretzel} and~\ref{thm:torus_alex}.

\subsection{Acknowledgments}

This project was completed primarily during the summer of 2023 as part of the Polymath Jr. Virtual REU, and the authors are grateful to the Polymath Jr. organizers for providing the opportunity to carry out this research.  We also thank Stefan Friedl and Filip Misev for helpful conversations and Taizo Kanenobu and Tomoyuki Yasuda for a helpful email correspondence.
JM was supported by NSF grant DMS-2006029; AZ was supported by NSF grants DMS-2005518 and DMS-2405301 and a Simons Fellowship.

\section{Preliminaries}\label{sec:prelim}

We work in the smooth category throughout.
We give most definitions a brief treatment here, and we refer the reader to~\cite{FMZ} for additional details.  As in Section~2 of~\cite{FMZ}, we begin with upper bounds.  Aside from explicit constructions, we get upper bounds from symmetric union presentations.  A \emph{symmetric union presentation} for a knot $K$ is a diagram $D^*$ for $K$ such that
\begin{enumerate}
\item $D^*$ has a vertical axis of symmetry $L$, outside of a neighborhood of which $D^*$ has reflection symmetry;
\item $D^*$ has some number (possibly zero) of crossings contained in $L$; and
\item Exactly two horizontal strands of $D^*$ cross $L$.
\end{enumerate}
Given a symmetric union presentation $D^*$, performing the vertical smoothing of all crossings on the axis of symmetry $L$ converts $D^*$ to $D \# \overline{D}$, where $\overline{D}$ is the mirror image of $D$, and we call the diagram $D$ the \emph{partial knot diagram} associated to $D^*$.  See Figure~\ref{fig:symm-ex} for an example of a symmetric union presentation $D^*$ for the knot $12n_{605}$ with partial diagram $D$, a six-crossing diagram for the trefoil.

\begin{figure}[h!]
    \centering
    \includegraphics[height=4cm]{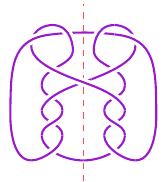} \qquad \qquad
       \includegraphics[height=4cm]{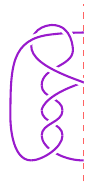} \quad\qquad
          \includegraphics[height=4cm]{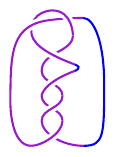} 
    \caption{A symmetric union presentation $D^*$ for $12n_{605}$ (left), and the corresponding partial diagram $D$ (right).}
    \label{fig:symm-ex}
\end{figure}

Every knot that admits a symmetric union presentation is ribbon, and it is an open problem to due Lamm whether every ribbon knot admits a symmetric union presentation.  See~\cite{lamm1,lamm2,lamm} for further details.  We can use symmetric union presentations to bound $r(K)$ from above.

\begin{lemma}\label{lem:symm}\cite[Lemma 2.1]{FMZ}
Suppose that $K$ admits a symmetric union presentation $D^*$ with partial diagram $D$, and the horizontal strands crossing the axis of symmetric are adjacent to $\ell$ consecutive undercrossings in $D$.  Then $r(K) \leq c(D) - \ell$.
\end{lemma}

To see Lemma~\ref{lem:symm} in action, consider the symmetric union presentation for the knot $12n_{605}$ shown in Figure~\ref{fig:symm-ex}.  For this example, the partial diagram $D$ is a 6-crossing diagram for the trefoil, and there are two undercrossings adjacent to the horizontal strands.  Thus, by Lemma~\ref{lem:symm}, we have $r(12n_{605}) \leq 4$.  At left in Figure~\ref{fig:symm-ex2}, we see a ribbon disk $\D$ for $12n_{605}$ induced by the presentation $D^*$ with $c(D)$ crossings, and at right we see a ribbon disk $\D'$ for the same knot with $\ell(D)$ fewer crossings.

\begin{figure}[h!]
    \centering
    \includegraphics[width=0.2\textwidth]{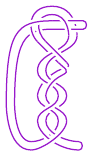} \qquad \qquad
       \includegraphics[width=0.2\textwidth]{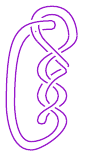}
        \caption{At left, a ribbon disk $\D$ induced by the symmetric union presentation $D^*$ for $12n_{605}$ shown in Figure~\ref{fig:symm-ex}, where $r(\D) = c(D)$.  At right, a simplified ribbon disk $\D'$, where $r(\D') = c(D) - \ell(D)$.}
    \label{fig:symm-ex2}
\end{figure}

\begin{remark}
Technically, Lemma 2.1 in~\cite{FMZ} is stated for consecutive undercrossings adjacent to \emph{one} horizontal strand, but the same proof works when considering both horizontal strands, as demonstrated by the example in Figure~\ref{fig:symm-ex2}.
Additionally, Lemma~\ref{lem:symm} also applies if we consider $\ell$ consecutive \emph{overcrossings} adjacent to both horizontal strands, since we can revolve any symmetric union presentation about its axis of symmetry to change these overcrossings to undercrossings.
\end{remark}

Turning now to lower bounds on ribbon numbers, our most basic tool relates ribbon number to knot genus.

\begin{lemma}\label{lem:genus}\cite[Lemma 2.3]{FMZ}
For any ribbon knot $K$, we have $r(K) \geq g(K)$.
\end{lemma}

Recall that $\Rr_r = \{\Delta_K(t) : r(K) \leq r\}$.  By construction, $\Rr_r \subset \Rr_{r+1}$.  It is well-known that no non-trivial knot satisfies $r(K) = 1$; moreover, $\Rr_2$ and $\Rr_3$ have been enumerated.

\begin{proposition}\label{prop:r2r3}\cite[Propositions 5.11 and 5.12]{FMZ}
	The set $\Rr_2$ contains two elements, and the set $\Rr_3$ contains ten elements, as shown in Table~\ref{table:R3}.
\end{proposition}

\begin{table}[!ht]
    \centering
    \begin{tabular}{|l|l|l|l|}
    \hline
        Det & Alexander Polynomial & Alexander Half-Polynomial & Note \\ \hline
        1 & $1$ & $1$ & $\mathfrak R_0$ \\ \hline
        1 & $1 - t - t^2 + 3t^3 - t^4 - t^5 + t^6$ & $1 - t^2 + t^3$ &  \\ \hline
        1 & $1 - 3t^2 + t^4$ & $1 + t - t^2$ &  \\ \hline
        9 & $1 - 2t + 3t^2 - 2t^3 + t^4$ & $1 - t + t^2$ & $\mathfrak R_2$ \\ \hline
        9 & $2 - 5t + 2t^2$ & $-1 + 2t$ & $\mathfrak R_2$ \\ \hline
        25 & $2 - 6t + 9t^2 - 6t^3 + 2t^4$ & $2 - 2t + t^2$ &  \\ \hline
        25 & $1 - 3t + 5t^2 - 7t^3 + 5t^4 - 3t^5 + t^6$ & $1 - t + 2t^2 - t^3$ &  \\ \hline
        25 & $1 - 6t + 11t^2 - 6t^3 + t^4$ & $-1 + 3t - t^2$ &  \\ \hline
        49 & $3 - 12t + 19t^2 - 12t^3 + 3t^4$ & $3 - 3t + t^2$ &  \\ \hline
        49 & $1 - 5t + 11t^2 - 15t^3 + 11t^4 - 5t^5 + t^6$ & $1 - 2t + 3t^2 - t^3$ &  \\ \hline
    \end{tabular}
	\caption{The 10 members of $\mathfrak R_3$, ordered by determinant}
	\label{table:R3}
\end{table}

The Alexander polynomial $\Delta_K(t)$ is determined up to multiplication by a unit in $\Z[t,t^{-1}]$, and if $K$ is a ribbon knot, there is a representative $\Delta_K(t)$ that factors as $f(t) \cdot f(t^{-1})$ for some $f(t) \in \Z[t]$.  In this case, we choose such an $f(t)$ and call it the \emph{Alexander half-polynomial} of $K$.
Alexander half-polynomials are determined up to substitution of $t^{-1}$ for $t$ and/or multiplication by a unit.

Following~\cite{FMZ}, we use a tool called a \emph{ribbon code} to organize the information in a ribbon disk.  Before defining ribbon codes, we first define disk-band presentations.
Let $\D$ be a ribbon disk.
A \emph{disk-band presentation} $(D,B)$ for $\D$ consists of a collection $D$ of embedded disks in $S^3$ along with a collection $B$ of embedded bands attached to $\pd D$ such that $\D = D \cup B$.
Note that the ribbon self-intersections of $\D$ occur precisely where the bands $B$ intersect the interiors of the disks $D$.
Every ribbon disk can be represented by a disk-band presentation.  See~\cite{FMZ} for additional details.

Now, suppose that $(D,B)$ is a disk-band presentation for $\D$ equipped with a normal orientation.  We associate a marked graph $\Gamma$ to $(D,B)$ in the following way:  Each disk $D_i \subset D$ gives rise to a vertex $v_i$, and each band $B_k$ attached to disks $D_i$ and $D_j$ gives rise to an edge $e_k$ connecting the corresponding vertices $v_i$ and $v_j$.  For each arc $\A$ of a ribbon intersection of the band $B_k$ with a disk $D_{\ell}$, the arc $\A$ cuts the preimage of $\D$ into two components, only one of which contains $D_{\ell}$, and we define the \emph{local direction} at the ribbon intersection to point along $B_k$ in the direction of the component containing $D_{\ell}$.
Finally, corresponding to each such ribbon intersection, we add a marking to the edge $e_{k}$ labeled $+\ell$ (respectively, $-\ell)$ if the local direction and normal direction of $D_{\ell}$ agree (respectively, disagree) at $\A$.  An example disk-band presentation and corresponding ribbon code are shown in Example~\ref{ex:code}.
We can also express a ribbon code as a tuple of vectors of integers:  If $e_k$ connects $v_{i}$ to $v_{j}$ and has markings $\pm \ell_1,\dots, \pm \ell_n$, we express $e_k$ as the vector $[i,\pm \ell_1 ,\dots,\pm \ell_n,j]$.  The tuple is then a list of all edge vectors.
See Example~\ref{ex:code}.
Howie gave a similar combinatorial treatment of ribbon disks~\cite{howie_asphericity}; see Remark~\ref{rmk:howie}.

\begin{example}\label{ex:code}
In Figure~\ref{fig:code}, we see a disk-band presentation for a ribbon disk $\D$ whose boundary is the knot $12n_{288}$.  This presentation consists of three disks, two bands, and four ribbon intersections, and gives rise to the ribbon code shown. The edge $e_1$ has vector $[1,-3,2,-3,2]$, while the edge $e_2$ has vector $[2,1,3]$, so that the tuple corresponding to this ribbon code is $([1,-3,2,-3,2],[2,1,3])$.
\begin{figure}[h!]
    \centering
    \includegraphics[width=0.6\textwidth]{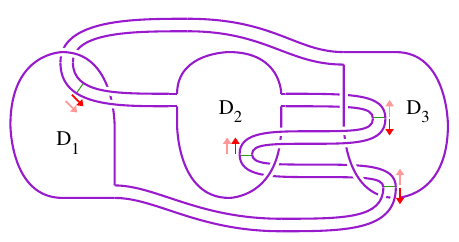} \\
       \includegraphics[width=0.5\textwidth]{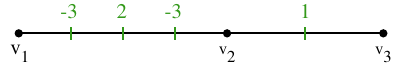}
        \caption{A ribbon disk and corresponding ribbon code.  Local directions are shown in pink and the normal orientation for $\D$ comes out of the page, shown in red.}
    \label{fig:code}
\end{figure}
\end{example}

Reindexing the disks in a disk-band presentation gives rise to an isomorphism of ribbon codes:  Two ribbon codes $\Gamma$ and $\Gamma'$ are \emph{isomorphic} if the underlying graphs are isomorphic via a graph isomorphism $\varphi\colon \Gamma \rightarrow \Gamma'$ that induces a bijection between the markings of $\Gamma$ and $\Gamma'$ such that, whenever $\varphi(v_i) = v'_{\sigma(i)}$ and a marking $\mu_{\ell}$ in $\Gamma$ is labeled $\pm i$, the marking $\varphi(\mu_{\ell})$ in $\Gamma'$ is labeled $\pm \sigma(i)$.
In~\cite{FMZ}, the authors showed

\begin{proposition}\label{prop:alex}\cite[Corollary 4.5]{FMZ}
Suppose $K$ and $K'$ bound ribbon disks $\D$ and $\D'$ with disk-band presentations yielding isomoprhic ribbon codes.  Then $\Delta_K(t) = \Delta_{K'}(t)$.
\end{proposition}

As a result of this proposition, we can unambiguously define $\Delta_{\Gamma}(t) = \Delta_K(t)$, where $K$ is any knot bounding a disk $\D$ with disk-band presentation $(D,B)$ corresponding to $\Gamma$.  The \emph{ribbon number} $r(\Gamma)$ of a ribbon code $\Gamma$ is the number of markings contained in the edge set of $\Gamma$, and by construction, if  $(D,B)$ is a disk-band presentation for $\D$ with associated ribbon code $\Gamma$, we have $r(\D) = r(\Gamma)$.  Another useful measure of complexity is the \emph{fusion number} $\F(\Gamma)$, which is defined to be the number of edges in $\Gamma$ (and which agrees with the number $\F(D,B)$ of bands in $(D,B)$).  In~\cite{FMZ}, the authors deemed a ribbon code $\Gamma$ \emph{reducible} if one of the following occurs
\begin{enumerate}
\item An edge contains no marking,
\item An edge contains consecutive markings labeled $i$ and $-i$,
\item A marking nearest to a vertex $v_i$ is labeled $\pm i$, or
\item A vertex of valence one or two does not appear as the label on any marking.
\end{enumerate}
If a ribbon code is not reducible, it is \emph{irreducible}.  In~\cite{FMZ}, the authors proved

\begin{proposition}\label{prop:reduc}\cite[Proposition 5.9]{FMZ}
If a ribbon code $\Gamma$ is reducible, then there exists an irreducible ribbon code $\Gamma'$ such that $\Delta_{\Gamma'}(t) = \Delta_{\Gamma}(t)$ and $r(\Gamma') \leq r(\Gamma)$.
\end{proposition}

The crucial take-away here is that by Proposition~\ref{prop:reduc}, in order the determine the set $\Rr_4$ of all possible Alexander polynomials of knots with ribbon number at most four, we need only enumerate all irreducible ribbon codes $\Gamma$ such that $r(\Gamma) \leq 4$, a much more manageable task, and one we execute by exhaustion in the next section.  Another tool that helps us eliminate cases in a brute force argument is

\begin{lemma}\label{lem:negate}\cite[Lemma 5.8]{FMZ}
Suppose that $\Gamma$ and $\Gamma'$ have isomorphic underlying graphs and isomorphic markings, but such that the label of each marking of $\Gamma'$ is opposite that of $\Gamma$.  Then $\Delta_{\Gamma} = \Delta_{\Gamma'}$.
\end{lemma}

Sometimes an irreducible ribbon code can be broken up into simpler ribbon codes.
If $(D',B')$ and $(D'',B'')$ are two disk-band presentations, we can form a new disk-band presentation by taking the boundary connected sum along arcs $d'\subset \pd D'$ and $d''\subset \pd D''$.
We refer to any ribbon disk that has a disk-band presentation formed in this way (and to its corresponding ribbon code) as \emph{decomposable}.
A decomposable ribbon code can be characterized by the property that there is a vertex $v_i$ with degree at least two such that each marking $\mu_\ell$ is on the same side of $v_i$ as $v_{|\mu_\ell|}$ (for $\ell\not=i$).  See Figure~\ref{fig:decomp} for an example of a decomposable ribbon disk and ribbon code.  If $\Dd = \Dd'\natural\Dd''$ is decomposable, then $\Delta_\Dd = \Delta_{\Dd'}\Delta_{\Dd''}$, so in our efforts to determine $\Rr_4$, we can disregard decomposable ribbon codes, as long as we remember to include all possible products of two elements of $\Rr_2$.  If a ribbon code is not decomposable, it is \emph{indecomposable}.

\begin{figure}[h!]
\begin{subfigure}{.48\textwidth}
  \centering
  \includegraphics[width=.95\linewidth]{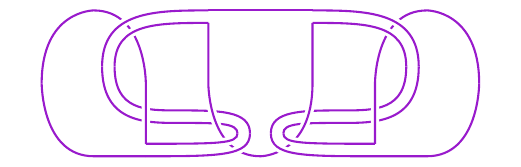}
\end{subfigure} \quad
\begin{subfigure}{.48\textwidth}
  \centering
  \includegraphics[width=.95\linewidth]{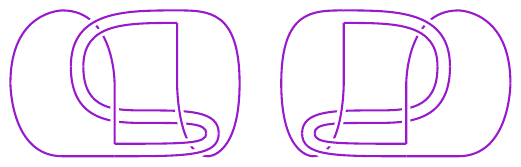}
\end{subfigure} \\
\begin{subfigure}{.48\textwidth}
  \centering
  \includegraphics[width=.95\linewidth]{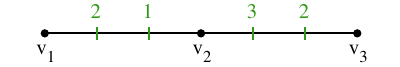}
\end{subfigure} \quad
\begin{subfigure}{.48\textwidth}
  \centering
  \includegraphics[width=.95\linewidth]{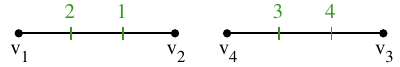}
\end{subfigure}
\caption{A decomposable ribbon disk and ribbon code (left) can be broken into summands (right).}
    \label{fig:decomp}
\end{figure}

To conclude this section, we introduce a new move that will help simplify our argument, a leaf isotopy.  Suppose $\Gamma$ is a ribbon code satisfying the following conditions:
\begin{enumerate}
\item The vertex $v_i$ has degree one (i.e. $v_i$ is a \emph{leaf} of $\Gamma$),
\item The marking $\mu_{\ell}$ adjacent to $v_i$ is labeled $\pm j$,
\item The vertex $v_j$ is adjacent to a marking $\mu_{\ell'}$ labeled $\pm i$, and
\item The marking $\mu_{\ell'}$ is the only marking labeled $\pm i$.
\end{enumerate}
Let $\Gamma'$ be the new ribbon code obtained from $\Gamma$ by removing the marking $\mu_{\ell}$ from near $v_i$ and placing it adjacent to the marking $\mu_{\ell'}$, so that $\mu_{\ell'}$ is between the vertex $v_j$ and the marking $\mu_{\ell}$ in $\Gamma'$, as shown in Figure~\ref{fig:leaf_isotopy}. 
We say that $\Gamma'$ is the result of a \emph{leaf isotopy} performed on~$\Gamma$ at vertex $v_i$.  

\begin{lemma}\label{lem:leaf}
Suppose that $\Gamma$ and $\Gamma'$ are related by a leaf isotopy.  Then $\Delta_{\Gamma} = \Delta_{\Gamma'}$.
\end{lemma}

\begin{proof}
It suffices to find a knot $K$ bounding disks $\D$ and $\D'$ yielding $\Gamma$ and $\Gamma'$, respectively.  If the conditions above are satisfied, we can construct the disks $\D$ and $\D'$ as shown in Figure~\ref{fig:leaf_isotopy}, noting that $\pd \D$ and $\pd \D'$ are isotopic.  If the sign of either marking is opposite what is shown in the figure, we can flip one or both disks, which results in a half-twist added to one or both bands; thus, the figure is still representative of all possible cases.
\end{proof}

\begin{figure}[h!]
\begin{subfigure}{.48\textwidth}
  \centering
  \includegraphics[width=.35\linewidth]{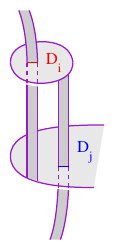}
\end{subfigure} 
\begin{subfigure}{.48\textwidth}
  \centering
  \includegraphics[width=.35\linewidth]{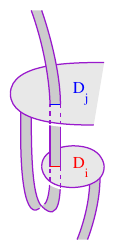}
\end{subfigure} \\
\begin{subfigure}{.48\textwidth}
  \centering
  \includegraphics[width=.45\linewidth]{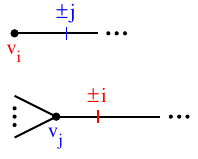}
\end{subfigure}%
\begin{subfigure}{.48\textwidth}
  \centering
  \includegraphics[width=.45\linewidth]{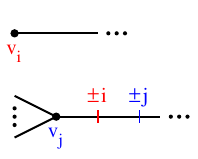}
\end{subfigure}
\caption{Disk $\D$ before (left) and disk $\D'$ after (right) a leaf isotopy, with corresponding ribbon codes.  Note that the vertex $v_i$ must be a leaf, and the label $\pm i$ must appear only once.}
    \label{fig:leaf_isotopy}
\end{figure}

\begin{remark}
In~\cite{FMZ}, the authors identified eight different irreducible ribbon codes with ribbon number three, but $\Rr_3 \setminus \Rr_2$ contains only seven Alexander polynomials because two codes, $([1,2,-3,2],[2,1,3])$ and $([1,-2,-3,2],[2,1,3])$, are related by a leaf isotopy: We can change
$$([\textcolor{red}{1},\textcolor{blue}{2},-3,2],[\textcolor{blue}{2},\textcolor{red}{1},3]) \rightarrow ([\textcolor{red}{1},-3,2],[\textcolor{blue}{2},\textcolor{red}{1},\textcolor{blue}{2},3]),$$
via a leaf isotopy at $v_1$.
The resulting code is isomorphic to $([1,2,3,2],[2,-1,3])$ via interchanging $3$ and $1$, which is related to $([1,-2,-3,2],[2,1,3])$ by Lemma~\ref{lem:negate}.
\end{remark}

\section{Determining $\mathfrak R_4$}
\label{sec:R4}

In this section, we determine the set $\mathfrak R_4$ consisting of all Alexander polynomials of ribbon knots that bound ribbon disks with four ribbon intersections.
The main result is the following.

\begin{proposition}
\label{prop:R4}
	If $K$ is a ribbon knot and $\Delta_K(t)\in\mathfrak R_4\setminus\mathfrak R_3$, then $\Delta_K(t)$ is one of the 46 entries in Table~\ref{table:R4}.
\end{proposition}

\begin{proof}
	The proof is by brute force.
	Suppose $K$ is a ribbon knot bounding a ribbon disk $\D$ with ribbon number 4.
	By Proposition~\ref{prop:reduc}, there is an irreducible ribbon code $\Gamma$ with ribbon number at most 4 such that $\Delta_{\Gamma} = \Delta_K$.  If $\Gamma$ has ribbon number less than 4, then $\Delta_K(t)\in\mathfrak R_3$, a contradiction.
	Thus, $\Gamma$ is an irreducible ribbon code with ribbon number 4.
	
	In Lemma~\ref{lem:structures}, below, we identify the 8 possible structures for $\Gamma$.
	In Lemma~\ref{lem:4codes}, we give a list that includes every indecomposable, irreducible ribbon code with ribbon number 4, though duplication is rampant.
	For each ribbon code on the list, using a Sage program, we employ Fox calculus~\cite{crowell-fox} to determine the Alexander polynomial of the ribbon code.
	Note that the decomposable irreducible ribbon codes with ribbon number 4 are exactly those polynomials which are the product of two elements of $\Rr_2$; these polynomials appear on the list.
	The resulting Alexander polynomials are presented in Tables~\ref{table:R3} and~\ref{table:R4}, depending on whether or not they occur in $\mathfrak R_3$.
	\end{proof}

\begin{table}[!ht]
    \centering
	\resizebox*{!}{.95\dimexpr\textheight-1\baselineskip\relax}{%
    \begin{tabular}{|l|l|l|l|}
    \hline
        Det & Alexander Polynomial & Alexander Half-Polynomial & 2-knot~\cite{yasuda:2-knot} \\ \hline
        1 & $1 - 4t + 4t^2 + 4t^3 - 11t^4 + 4t^5 + 4t^6 - 4t^7 + t^8$ & $-1 + 2t + t^2 - 2t^3 + t^4$ &  \\ \hline
        1 & $1 - 3t - t^2 + 7t^3 - t^4 - 3t^5 + t^6$ & $-1 + t + 2t^2 -t^3$ & $K_{25}$  \\ \hline
        1 & $1 - 3t + 2t^2 + 3t^3 - 7t^4 + 3t^5 + 2t^6 - 3t^7 + t^8$ & $1 - t + 2t^3 - t^4$ & $K_{24}$  \\ \hline
        1 & $1 - 2t + t^2 + 2t^3 - 5t^4 + 2t^5 + t^6 - 2t^7 + t^8$ & $1 + - t + t^2 + t^3 - t^4$ & $K_{28}$  \\ \hline
        1 & $1 - 2t^2 + 3t^4 - 2t^6 + t^8$ & $1 - t^2 + t^4$ & $K_{5}$  \\ \hline
        1 & $2 - 4t - 2t^2 + 9t^3 - 2t^4 - 4t^5 + 2t^6$ & $2 - 2t^2 + t^3$ &   \\ \hline
        1 & $2 - 3t - 2t^2 + 7t^3 - 2t^4 - 3t^5 + 2t^6$ & $1 - t - t^2 + 2t^3$ & $K_{29}$  \\ \hline
        1 & $2 - 5t^2 + 2t^4$ & $2 - t^2$ & $K_{6}$  \\ \hline
        9 & $1 - 4t + 3t^2 + 6t^3 - 13t^4 + 6t^5 + 3t^6 - 4t^7 + t^8$ & $1 - t - t^2 + 3t^3 - t^4$ &   \\ \hline
        9 & $1 - 3t + 2t^2 + 5t^3 - 11t^4 + 5t^5 + 2t^6 - 3t^7 + t^8$ & $-1 + t + 2t^2 - 2t^3 + t^4$ &   \\ \hline
        9 & $1 - 2t + 4t^3 - 7t^4 + 4t^5 - 2t^7 + t^8$ & $1 - t^2 + 2t^3 - t^4$ & $K_{19}$  \\ \hline
        9 & $1 - t - 3t^2 + 7t^3 - 3t^4 - t^5 + t^6$ & $1 + t - 2t^2 + t^3$ & $K_{12}$  \\ \hline
        9 & $1 - t - t^3 + 3t^4 - t^5 - t^7 + t^8$ & $1 - t^3 + t^4$ & $K_{1}$  \\ \hline
        9 & $1 - t + t^2 - 3t^3 + t^4 - t^5 + t^6$ & $1 + t^2 - t^3$ & $K_{3},K_{14}$  \\ \hline
        9 & $2 - 2t - 4t^2 + 9t^3 - 4t^4 - 2t^5 + 2t^6$ & $1 - 2t^2 + 2t^3$ & $K_{11}$  \\ \hline
        25 & $1 - 3t + 9t^3 - 15t^4 + 9t^5 - 3t^7 + t^8$ & $1 - 2t^2 + 3t^3 - t^4$ &   \\ \hline
        25 & $1 - 2t + 3t^2 - 4t^3 + 5t^4 - 4t^5 + 3t^6 - 2t^7 + t^8$ & $1 - t + t^2 - t^3 + t^4$ & $K_{2}$  \\ \hline
        25 & $1 - t - 7t^2 + 15t^3 - 7t^4 - t^5 + t^6$ & $1 + 2t - 3t^2 + t^3$ &   \\ \hline
        25 & $6 - 13t + 6t^2$ & $3 - 2t$ & $K_{4}$  \\ \hline
        49 & $1 - 4t + 6t^2 - 8t^3 + 11t^4 - 8t^5 + 6t^6 - 4t^7 + t^8$ & $-1 + 2t - t^2 + 2t^3 - t^4$ & $K_{57}$; \small see Remark~\ref{rem:yasuda}  \\ \hline
        49 & $1 - 3t + 6t^2 - 9t^3 + 11t^4 - 9t^5 + 6t^6 - 3t^7 + t^8$ & $1 - t + 2t^2 - 2t^3 + t^4$ & $K_{21}$  \\ \hline
        49 & $2 - 12t + 21t^2 - 12t^3 + 2t^4$ & $-2t + 4t - t^2$ & $K_{22}$  \\ \hline
        49 & $2 - 6t + 10t^2 - 13t^3 + 10t^4 - 6t^5 + 2t^6$ & $2 - 2t + 2t^2 - t^3$ & $K_{20}$  \\ \hline
        49 & $4 - 12t + 17t^2 - 12t^3 + 4t^4$ & $2 - 3t + 2t^2$ & $K_{23},K_{56}$  \\ \hline
        81 & $1 - 7t + 19t^2 - 27t^3 + 19t^4 - 7t^5 + t^6$ & $-1 + 4t - 3t^2 + t^3$ & $K_{61},K_{62}$  \\ \hline
        81 & $1 - 5t + 10t^2 - 15t^3 + 19t^4 - 15t^5 + 10t^6 - 5t^7 + t^8$ & $-1 + 2t - 2t^2 + 3t^3 - t^4$ & $K_{85}$  \\ \hline
        81 & $1 - 4t + 9t^2 - 16t^3 + 21t^4 - 16t^5 + 9t^6 - 4t^7 + t^8$ & $1 - t + 3t^2 - 3t^3 + t^4$ & $K_{59}$  \\ \hline
        81 & $1 - 4t + 10t^2 - 16t^3 + 19t^4 - 16t^5 + 10t^6 - 4t^7 + t^8$ & $1 - 2t + 3t^2 - 2t^3 + t^4$ & $K_{40},K_{43},K_{46},K_{83}$  \\ \hline
        81 & $2 - 9t + 18t^2 - 23t^3 + 18t^4 - 9t^5 + 2t^6$ & $2 - 3t + 3t^2 - t^3$ & $K_{41},K_{44},K_{49},K_{60}$  \\ \hline
        81 & $2 - 8t + 18t^2 - 25t^3 + 18t^4 - 8t^5 + 2t^6$ & $1 - 2t + 4t^2 - 2t^3$ & $K_{86}$  \\ \hline
        81 & $4 - 20t + 33t^2 - 20t^3 + 4t^4$ & $4 - 4t + t^2$ \textbf{or} $2-5t+2t^2$ & $K_{47},K_{48},K_{50},K_{51}$  \\ \hline
        81 & $6 - 20t + 29t^2 - 20t^3 + 6t^4$ & $3 - 4t + 2t^2$ & $K_{84}$  \\ \hline
        121 & $1 - 9t + 29t^2 - 43t^3 + 29t^4 - 9t^5 + t^6$ & $-1 + 5t - 4t^2 + t^3$ & $K_{72}$  \\ \hline
        121 & $1 - 6t + 15t^2 - 24t^3 + 29t^4 - 24t^5 + 15t^6 - 6t^7 + t^8$ & $-1 + 3t - 3t^2 + 3t^3 - t^4$ & $K_{69}$  \\ \hline
        121 & $1 - 5t + 14t^2 - 25t^3 + 31t^4 - 25t^5 + 14t^6 - 5t^7 + t^8$ & $1 - 2t + 4t^2 - 3t^3 + t^4$ & $K_{67},K_{70}$  \\ \hline
        121 & $2 - 12t + 28t^2 - 37t^3 + 28t^4 - 12t^5 + 2t^6$ & $2 - 4t + 4t^2 - t^3$ & $K_{71},K_{74}$  \\ \hline
        121 & $2 - 11t + 28t^2 - 39t^3 + 28t^4 - 11t^5 + 2t^6$ & $1 - 3t + 5t^2 - 2t^3$ & $K_{73}$  \\ \hline
        121 & $3 - 13t + 27t^2 - 35t^3 + 27t^4 - 13t^5 + 3t^6$ & $3 - 4t + 3t^2 - t^3$ & $K_{68}$  \\ \hline
        169 & $1 - 7t + 20t^2 - 35t^3 + 43t^4 - 35t^5 + 20t^6 - 7t^7 + t^8$ & $-1 + 4t - 4t^2 + 3t^3 - t^4$ & $K_{93},K_{94},K_{96}$  \\ \hline
        169 & $1 - 6t + 18t^2 - 36t^3 + 47t^4 - 36t^5 + 18t^6 - 6t^7 + t^8$ & $1 - 2t + 5t^2 - 4t^3 + t^4$ & $K_{98}$  \\ \hline
        169 & $1 - 6t + 19t^2 - 36t^3 + 45t^4 - 36t^5 + 19t^6 - 6t^7 + t^8$ & $1 - 3t + 5t^2 - 3t^3 + t^4$ & $K_{95}$  \\ \hline
        169 & $2 - 14t + 40t^2 - 57t^3 + 40t^4 - 14t^5 + 2t^6$ & $1 - 4t + 6t^2 - 2t^3$ & $K_{92}$  \\ \hline
        169 & $3 - 17t + 39t^2 - 51t^3 + 39t^4 - 17t^5 + 3t^6$ & $3 - 5t + 4t^2 - t^3$ & $K_{91},K_{97}$  \\ \hline
        225 & $1 - 8t + 26t^2 - 48t^3 + 59t^4 - 48t^5 + 26t^6 - 8t^7 + t^8$ & $-1 + 4t - 5t^2 + 4t^3 - t^4$ & $K_{109}$  \\ \hline
        225 & $1 - 7t + 24t^2 - 49t^3 + 63t^4 - 49t^5 + 24t^6 - 7t^7 + t^8$ & $1 - 3t + 6t^2 - 4t^3 + t^4$ & $K_{108}$  \\ \hline
        225 & $4 - 22t + 52t^2 - 69t^3 + 52t^4 - 22t^5 + 4t^6$ & $4 - 6t + 4t^2 - t^3$ & $K_{107}$ \\ \hline
   	\end{tabular}}
	\caption{The 46 members of $\mathfrak R_4\setminus\mathfrak R_3$, ordered by determinant}
	\label{table:R4}
\end{table}

Note that the fourth column of Table~\ref{table:R4} contains the ribbon 2-knot(s) whose Alexander polynomial agrees with the corresponding Alexander half-polynomial, using Yasuda's notation~\cite{yasuda:2-knot}.  See Remark~\ref{rem:yasuda} for further details.

\begin{lemma}
\label{lem:structures}
	Suppose $\Gamma$ is an irreducible ribbon code with ribbon number four.
	Then $\Gamma$ has one of the eight structures shown in Figure~\ref{fig:R4_structures}.
\end{lemma}

\begin{figure}[h!]
    \centering
    \includegraphics[width=0.75\textwidth]{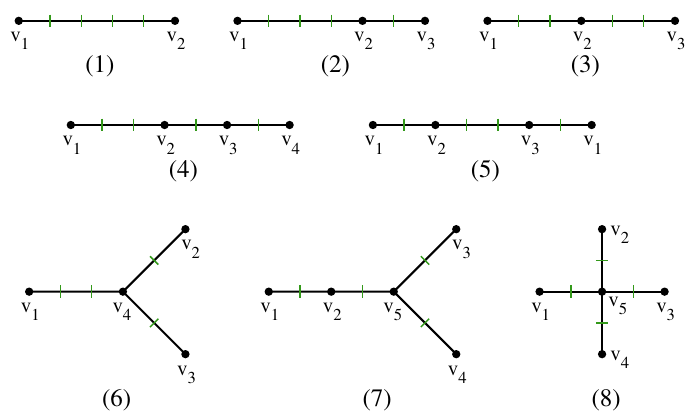}
    \caption{Ribbon code structures for ribbon number 4}
    \label{fig:R4_structures}
\end{figure}

\begin{proof}
	Suppose $\Gamma$ is an irreducible ribbon code with ribbon number four.
	The structure of $\Gamma$ is that of a tree with at least two vertices.
	Since $\Gamma$ is irreducible, it must have an edge marking on every edge.
	Since there are four edge markings, there are at most five vertices.
	Up to isomorphism, there are six different trees fitting these criteria, and there are eight ways to distribute the edge markings on the edges of these trees.
	The eight resulting structures are shown in Figure~\ref{fig:R4_structures}.
\end{proof}

In the next proof, we will use a number of leaf isotopies to convert ribbon codes from one structure to another.  We indicate these isotopies with the tuple notation, and we color-code vertices and markings for clarity, as in the definition of leaf isotopies in Section~\ref{sec:prelim}.  Sometimes a leaf isotopy will result in a reducible ribbon code $\Gamma'$ with an edge with no marking; in this case, we replace $\Gamma$ with an irreducible code by eliminating the edge and relabeling any affected markings with a single vertex.  For instance, the sequence below shows a leaf isotopy followed by a reduction, in which vertices $v_2$ and $v_3$ are combined into a single vertex, and any marking labeled $\pm 3$ is replaced with by $\pm 2$:
$$([\textcolor{blue}{1},\textcolor{red}{3},2,1,2],[2,\textcolor{blue}{1},\textcolor{red}{3}]) \rightarrow ([\textcolor{blue}{1},\textcolor{red}{3},\textcolor{blue}{1},2,1,2],[2,\textcolor{red}{3}]) \rightarrow ([1,2,1,2,1,2]).$$
If we combine, say, vertices $v_1$ and $v_2$, we shift all indices lower accordingly:
$$[\red{1},\blue{3},2],[2,4,2,3],[\blue{3},\red{1},4]) \rightarrow ([\red{1},2],[2,4,2,3],[\blue{3},\red{1},\blue{3},4]) \rightarrow ([1,3,1,2],[2,1,2,3]).$$
Note that the existence of a leaf isotopy is independent of the signs of the markings, and so we omit the signs for ease of notation in the above and in future instances of describing leaf isotopies.

\begin{lemma}
\label{lem:4codes}
	There are at most 118 inequivalent indecomposable, irreducible ribbon codes with ribbon number four; they are listed in Tables~\ref{table:rc1-2},~\ref{table:rc3-5}, and~\ref{table:rc6-8}.
	Organizing these by structure, and after eliminating some redundancies with leaf isotopies, we have:
	\begin{enumerate}
		\item Structure 1 admits at most 10 inequivalent indecomposable, irreducible ribbon codes.
		\item Structure 2 admits at most 24 inequivalent indecomposable, irreducible ribbon codes.
		\item Structure 3 admits at most 20 inequivalent indecomposable, irreducible ribbon codes.
		\item Structure 4 admits at most 16 inequivalent indecomposable, irreducible ribbon codes.
		\item Structure 5 admits at most  8 inequivalent indecomposable, irreducible ribbon codes.
		\item Structure 6 admits at most  28 inequivalent indecomposable, irreducible ribbon codes.
		\item Structure 7 admits at most  8 inequivalent indecomposable, irreducible ribbon codes.
		\item Structure 8 admits at most  4 inequivalent indecomposable, irreducible ribbon codes.
	\end{enumerate}
\end{lemma}

These ribbon codes are listed in Tables~\ref{table:rc1-2},~\ref{table:rc3-5}, and~\ref{table:rc6-8}, where they are organized by structure based on the cases occurring in the proof.

\begin{proof}
	We consider the eight structures in turn.
	
	\textbf{(Structure 1)} Suppose $\Gamma$ is an irreducible ribbon code with Structure 1.
	Let the tuple  $(\mu_1, \mu_2, \mu_3, \mu_4)$ denote the tuple of four marking labels on the unique edge, in order from $v_1$ to $v_2$.
	Since $\Gamma$ is irreducible, we must have $\mu_1=\pm2$ and $\mu_4=\pm1$, with no consecutive labels differing only in sign.
	After potentially relabeling the vertices, we can assume there are at least as many edge markings of $\pm2$ as there are of $\pm1$.
	After potentially negating all the edge markings, we can assume without loss of generality that $\mu_1=2$.
	If there are three negative markings, then tuple of labels is $(2,-1,-2,-1)$, for which negating all the edge markings and relabeling the vertices gives the tuple $(2,1,2,-1)$.  Thus, we can assume that at most two markings are negative.
	This leaves the following 10 possibilities for the edge markings: $(2,2,2,\pm 1)$, $(2,2,\pm1, \pm1)$, $(2,1,2,\pm 1)$, $(2,1,-2,\pm 1)$, or $(2,-1,2,\pm 1)$.

	In total, we have 10 ribbon codes of Structure 1.

	\textbf{(Structure 2)} Suppose $\Gamma$ is an irreducible ribbon code with Structure 2.
	Since $\Gamma$ is irreducible, after potentially negating all edge markings, we can assume that the marking on edge $\{v_2,v_3\}$ is $+1$.
	Let $(\mu_1, \mu_2, \mu_3)$ denote the three markings on edge $\{v_1,v_2\}$, in order from $v_1$ to $v_2$.
	
	\emph{Case 2.1:} There cannot be two instances of 1 among the $|\mu_i|$, since this would violate the rule that every vertex index appear as an edge marking.
	
	\emph{Case 2.2:} Suppose there are two instances of 2 among the $|\mu_i|$.
	Then, we must have $(|\mu_1|,|\mu_2|,|\mu_3|) = (2, 2, 3)$, where $\mu_1$ and $\mu_2$ have the same sign by irreducibility.
	The four choices for signs give 4 ribbon codes.
	
	\emph{Case 2.3:} Suppose there are two instances of 3 among the $|\mu_i|$.
	Then, we have the following two possibilities for $(|\mu_1|,|\mu_2|,|\mu_3|)$: $(2, 3, 3)$ or $(3, 2, 3)$.
	In the first case, there are four choices for signs.
	In the second case, there is a leaf isotopy
	\[ ([\red{1},\blue{3},2,3,2),(2,\red{1},\blue{3})] \rightarrow ([\red{1},2,3,2],[2,\blue{3},\red{1},\blue{3}]),\]
	which falls into Case 3.1 counted below.
	This gives 4 ribbon codes.
	
	\emph{Case 2.4:} If all the $|\mu_i|$ are distinct, then we have the following three possibilities for $(|\mu_1|,|\mu_2|,|\mu_3|)$: $(2,1,3)$, $(2,3,1)$, or $(3,2,1)$.
	In the third case, the code is reducible via a leaf isotopy at $v_3$ and need not be counted:
$$([\textcolor{blue}{1},\textcolor{red}{3},2,1,2],[2,\textcolor{blue}{1},\textcolor{red}{3}]) \rightarrow ([\textcolor{blue}{1},\textcolor{red}{3},\textcolor{blue}{1},2,1,2],[2,\textcolor{red}{3}]) \rightarrow ([1,2,1,2,1,2]).$$
In each other case, there are eight choices for signs.
	This gives 16 ribbon codes.
	
	In total, we have 24 ribbon codes of Structure 2.
	
	\textbf{(Structure 3)} Suppose $\Gamma$ is an irreducible ribbon code with Structure 3.
	Let the tuple $(\mu_1,\mu_2,\mu_3,\mu_4)$ denote the edge markings, from left to right, with $v_2$ lying between $\mu_2$ and $\mu_3$.
	There must be at least one instance of 2 among the $|\mu_i|$.
	After potentially negating all edge markings and permuting the vertex labels, we can assume that $\mu_1=2$.
	We consider the number of instances of 2 among the $|\mu_i|$.
	
	\emph{Case 3.1:} First assume that there is only the one instance of 2 among the $|\mu_i|$.
	In this case, we must have $|\mu_4|=1$.
	The possibilities for $(|\mu_2|,|\mu_3|)$ are $(3,1)$, $(1,3)$, and $(3,3)$.
	This gives 20 ribbon codes: four from the first case and eight from each other cases.
	
	\emph{Case 3.2:} Now assume there are two instances of 2 among the $|\mu_i|$, so we must have $|\mu_4|=2$ and $\{|\mu_2|,|\mu_3|\} = \{1,3\}$.
	If $(|\mu_2|,|\mu_3|) = (1,3)$, then the ribbon code is the decomposable code shown in Figure~\ref{fig:decomp} and need not be counted.	 If $(|\mu_2|,|\mu_3|) = (3,1)$, then the ribbon code is reducible via a leaf isotopy at $v_1$ (or $v_3$) to a ribbon code of Structure 2 and need not be counted:
	$$([\textcolor{red}{1},\textcolor{blue}{2},3,2],[\textcolor{blue}{2},\textcolor{red}{1},2,3]) \rightarrow ([\textcolor{red}{1},3,2],[\blue{2},\red{1},\blue{2},2,3]).$$
	
	In total, we have 20 ribbon codes of Structure 3.

	\textbf{(Structure 4)} Suppose $\Gamma$ is an irreducible ribbon code with Structure 4.
	Let the tuple $(\mu_1,\mu_2,\mu_3,\mu_4)$ denote the edge markings, with $\mu_i$ adjacent to $v_i$.
	Note that the $|\mu_i|$ must be distinct.
	We must have $|\mu_4|\in\{1,2\}$, and we can assume that $\mu_4$ is positive.
	
	\emph{Case 4.1:} If $\mu_4=1$, then we must have $|\mu_1|=2$, $|\mu_2|=3$, and $|\mu_3|=4$.
	This gives 8 ribbon codes.
	
	\emph{Case 4.2:} If $\mu_4=2$, then $(|\mu_1|,|\mu_2|,|\mu_3|)$ is either $(3,1,4)$, $(3,4,1)$, or $(4,3,1)$.
	In the first two cases, the code is subject to a leaf isotopy at $v_4$ reducing it to a code of Structure 3 or Structure 2, respectively, and need not be counted:
	$$([1,3,1,2],[\blue{2},\red{4},3],[3,\blue{2},\red{4}]) \rightarrow ([1,3,1,2],[\blue{2},\red{4},\blue{2},3],[3,\red{4}]) \rightarrow ([1,3,1,2],[2,3,2,3]),$$
	and
	$$([1,3,\red{4},\blue{2}],[2,1,3],[3,\blue{2},\red{4}]) \rightarrow ([1,3,\blue{2},\red{4},\blue{2}],[2,1,3],[3,\red{4}]) \rightarrow ([1,3,2,3,2],[2,1,3]).$$
	In the third case, there are 8 choices for signs.
	This gives 8 ribbon codes.

	In total, we have 16 ribbon codes of Structure 4.

	\textbf{(Structure 5)} Suppose $\Gamma$ is an irreducible ribbon code with Structure 5.
	Let the tuple $(\mu_1,\mu_2,\mu_3,\mu_4)$ denote the edge markings, with $\mu_i$ adjacent to $v_i$.
	Note that the $|\mu_i|$ must be distinct.
	We must have $|\mu_4|\in\{1,2\}$, and we can assume that $\mu_4$ is positive.
	
	\emph{Case 5.1:} Assume $\mu_4=1$.
	We must have $|\mu_3|=2$.
	If $|\mu_1|=4$, then the code is subject to a leaf isotopy at $v_4$ that reduces it to a ribbon code of Structure 3 and need not be counted:
	$$([\blue{1},\red{4},2],[2,3,2,3],[3,\blue{1},\red{4}]) \rightarrow ([\blue{1},\red{4},\blue{1},2],[2,3,2,3],[3,\red{4}]) \rightarrow ([1,3,1,2],[2,3,2,3]).$$
	If $|\mu_1|=3$, then the code is subject to a leaf isotopy at $v_1$ that reduces it to a ribbon code of Structure 3 and need not be counted:
	$$[\red{1},\blue{3},2],[2,4,2,3],[\blue{3},\red{1},4]) \rightarrow ([\red{1},2],[2,4,2,3],[\blue{3},\red{1},\blue{3},4]) \rightarrow ([1,3,1,2],[2,1,2,3]).$$
	
	\emph{Case 5.2:} Assume $\mu_4=2$.
	If $|\mu_1|=4$ or $|\mu_2|=4$, then the code is subject to a leaf isotopy at $v_4$ that reduces it to a ribbon code of Structure 3 or Structure 2, respectively, and need not be counted:
	$$([1,\red{4},\blue{2}],[2,3,1,3 ],[3,\blue{2},\red{4}]) \rightarrow ([1,\blue{2},\red{4},\blue{2}],[2,3,1,3],[3,\red{4}]) \rightarrow ([1,2,3,2],[2,3,1,3])$$
	or
	$$([1,3,2],[\blue{2},\red{4},1,3],[3,\blue{2},\red{4}]) \rightarrow ([1,3,2],[\blue{2},\red{4},\blue{2},1,3],[3,
	\red{4}]) \rightarrow ([1,3,2],[2,3,2,1,3]).$$
	Thus, we can assume that $|\mu_3|=4$, which implies that $|\mu_2|=1$ and $|\mu_1|=3$.
	There are 8 ribbon codes of this type.
	
	In total, we have 8 ribbon codes of Structure 5.

	\textbf{(Structure 6)} Suppose $\Gamma$ is an irreducible ribbon code with Structure 6.
	Let the tuple $(\mu_1,\mu_2,\mu_3,\mu_4)$ denote the edge markings, with $\mu_i$ adjacent to $v_i$.  First, we consider whether or not some edge marking is $\pm4$.
	
	\emph{Case 6.1:} If $\pm4$ appears as an edge marking, then we must have $|\mu_1|=4$, and we can assume that $\mu_1$ is positive.
	We consider whether or not $|\mu_4|=1$.
	
	\emph{Case 6.1.1:} Suppose $|\mu_4|=1$.
	This implies that $|\mu_2|=3$ and $|\mu_3|=2$.
	In this case, the ribbon code is subject to a leaf isotopy at $v_2$ (or at $v_3$) that reduces it to a ribbon code of Structure 3 and need not be counted:
	$$([1,4,1,4],[4,\red{3},\blue{2}],[4,\blue{2},\red{3}]) \rightarrow ([1,4,1,4],[4,\blue{2},\red{3},\blue{2}],[4,\red{3}]) \rightarrow ([1,3,1,3],[3,2,3,2]).$$
	
	\emph{Case 6.1.2:} Suppose $|\mu_4|\not=1$.
	This implies that $|\mu_4|=2$ or $|\mu_4|=3$.
	After potentially permuting the vertex labels, we can assume $|\mu_4|=2$. 
	This implies that $(|\mu_4|,|\mu_2|,|\mu_3|) = (2,3,1)$, yielding 8 ribbon codes.
	
	\emph{Case 6.2:} If $\pm4$ does not appear as an edge marking, then there are two instances of some edge marking (up to sign).
	We consider which edge marking occurs twice (up to sign).
	
	\emph{Case 6.2.1:} Suppose there are two instances of $\pm 1$.
	If $|\mu_2|=|\mu_3| = 1$, the the ribbon code is subject to a leaf isotopy at the vertex $v_{|\mu_1|}$ that reduces it to a ribbon code of Structure~2 and need not be counted: 
	$$([\blue{1},\red{2},3,4],[4,\blue{1},\red{2}],[4,1,3]) \rightarrow ([\blue{1},\red{2},\blue{1},3,4],[4,\red{2}],[4,1,3] \rightarrow ([1,2,1,3,2],[2,1,3]).$$
	If only one of $|\mu_2|$ and $|\mu_3|$ is 1, then up to relabeling the vertices, $(|\mu_1|,|\mu_2|,|\mu_3|,|\mu_4|) = (3,1,2,1)$, and there are 8 possible ribbon codes.
	
	\emph{Case 6.2.2:} Suppose there are two instances of $\pm2$.  The possibilities for $(|\mu_1|,|\mu_2|,|\mu_3|,|\mu_4|)$ are: $(2,3,1,2)$, $(2,1,2,3)$, $(2,3,2,1)$, and $(3,1,2,2)$.
	In the second case, the ribbon code admits a leaf isotopy at $v_1$ that transforms it into a ribbon code of the same structure that, after relabeling the vertices falls in Case 6.2.1 and need not be counted: 
	$$([\red{1},\blue{2},3,4],[4,\red{1},\blue{2}],[4,2,3]) \rightarrow ([\red{1},3,4],[4,\blue{2},\red{1},\blue{2}],[4,2,3]).$$
	In the third case, the ribbon code admits a leaf isotopy at $v_2$ that reduces it to a ribbon code of Structure~3 and need not be counted:
	$$([1,2,1,4],[4,\red{3},\blue{2}],[4,\blue{2},\red{3}]) \rightarrow ([1,2,1,4],[4,\blue{2},\red{3},\blue{2}],[4,\red{3}]) \rightarrow ([1,2,1,3],[3,2,3,2]).$$
	In the first case, $\mu_1$ and $\mu_4$ must have the same sign, which can be assumed to be positive after potentially negating all edge labels, so there are 4 ribbon codes of this sort.
	In the fourth case, there are 8 ribbon codes.
		
	\emph{Case 6.2.3:} If there are two instances of $\pm3$, then we can relabel the vertices so that there are two instances of $\pm2$, which we have already counted.
	
	In total, we have 28 ribbon codes of Structure 6.

	\textbf{(Structure 7)} Suppose $\Gamma$ is an irreducible ribbon code with Structure 7.
	Let the tuple $(\mu_1,\mu_2,\mu_3,\mu_4)$ denote the edge markings, with $\mu_i$ adjacent to $v_i$.
	Note that the $|\mu_i|$ must be distinct and not 5.
	We must have $|\mu_1|\in\{3,4\}$, and after potentially permuting and negating the vertex labels, we can assume that $\mu_1=4$.
	We consider whether or not $|\mu_4|=1$.
	
	\emph{Case 7.1:} Assume $|\mu_4|=1$.
	Then, the ribbon code is subject to a leaf isotopy at $v_1$ (or at $v_4$) that transforms it into a ribbon code of Structure 6 (or of Structure 4), and it need not be counted:
{\small{	$$([\blue{1},\red{4},2],[2,3,5],[5,2,3],[5,\blue{1},\red{4}]) \rightarrow ([\blue{1},\red{4},\blue{1},2],[2,3,5],[5,2,3],[5,\red{4}]) \rightarrow ([1,4,1,2],[2,3,4],[4,2,3]).$$}}


	\emph{Case 7.2:} Assume $|\mu_4|=2$.
	Then, $(|\mu_2|, |\mu_3|)$ is $(3,1)$.
	In this case, the ribbon code is subject to a leaf isotopy at $v_4$ that transforms it into a ribbon code of Structure 4 and need not be counted.
{\small{	$$[1,\red{4},\blue{2}],[2,3,5],[5,1,3],[5,\blue{2},\red{4}] \rightarrow ([1,\blue{2},\red{4},\blue{2}],[2,3,5],[5,1,3],[5,\red{4}]) \rightarrow ([1,2,4,2],[2,3,4],[4,1,3]).$$}}
	
	\emph{Case 7.2:} Assume $|\mu_4|=3$.
	Then, $(|\mu_2|, |\mu_3|)$ is $(1,2)$.
	In this case, there are eight choices for signs, yielding 8 ribbon codes.
	
	In total, there are 8 ribbon codes of Structure 7.
	
	\textbf{(Structure 8)} Suppose $\Gamma$ is an irreducible ribbon code with Structure 8.
	Let the tuple $(\mu_1,\mu_2,\mu_3,\mu_4)$ denote the edge markings, with each $\mu_i$ adjacent to $v_i$.
	Note that the $|\mu_i|$ must be distinct and not 5.
	After potentially permuting and negating the vertex labels, we can assume that $\mu_1=2$.
	We consider whether or not $|\mu_2|=1$.
	
	\emph{Case 8.1:} Suppose that $|\mu_2|=1$.
	Then, $|\mu_3|=4$ and $|\mu_4|=3$, and the code is decomposable at $v_5$ and need not be counted:
	$$([1,2,\orange{5}],[2,1,\orange{5}],[3,4,\orange{5}],[4,3,\orange{5}])\rightarrow ([1,2,\orange{5}],[2,1,\orange{5}])\#([3,4,\orange{5}],[4,3,\orange{5}]).$$
	
	\emph{Case 8.2:} Suppose that $|\mu_2|\not=1$.
	After potentially permuting the vertex labels, we can assume $(|\mu_2|,|\mu_3|,|\mu_4|) = (3,4,1)$.
	After potentially permuting the vertex labels and negating all edge markings, we can assume that at most two of the $\mu_i$ are negative; in particular, we have only four options for $(\mu_1,\mu_2,\mu_3,\mu_4)$: $(2,3,4,\pm1)$, $(2,3,-4,-1)$, or $(2,-3,4,-1)$.
	This gives 4 ribbon codes.

	In total, there are 4 ribbon codes of Structure 8.
\end{proof}

\section{Tabulation}\label{sec:tab}

In this section, we tabulate ribbon numbers for prime 12-crossing knots.  Our upper bounds come primarily from an application of Lemma~\ref{lem:symm} to the symmetric union presentations appearing in~\cite{lamm}, which we reference in the table as Lemma~\ref{lem:symm}.  When we do not use symmetric union presentations, we will refer to the ribbon disks shown in Figures~\ref{fig:ribbonA} and~\ref{fig:ribbonB} below or to a specific figure in a reference.  Knots in these figures were verified by SnapPy~\cite{snappy}.  Lower bounds are derived primarily from the genus obstruction in Lemma~\ref{lem:genus} and the Alexander polynomial obstructions in Propositions~\ref{prop:r2r3} and~\ref{prop:R4}.

We will also cite a proposition due to Mizuma and Tsutsumi, which uses the \emph{crosscap number} $\gamma(K)$ of a knot $K$, the minimal non-orientable genus of any surface bounded by $K$.

\begin{proposition}\cite{MT}\label{prop:cross}
Suppose $r(K) = 2$.  Then either $g(K) = 1$ or $\gamma(K) \leq 2$.
\end{proposition}

\begin{proof}
Suppose $\D$ is a ribbon disk for $K$ such that $r(\D)  = 2$.  By Lemma 2.5 of~\cite{FMZ}, there exists a disk-band presentation $(D,B)$ for $\D$ such that $\F(D,B) = 1$, and we can construct an embedded surface $S$ bounded by $K$ by removing two disk neighborhoods of the ribbon intersections of $\D$ and connecting their boundaries with a tube that encloses a section of the band.  Either $S$ is orientable, in which case $g(S) = 1$, or $S$ is non-orientable, in which case $\gamma(S) = 2$.
\end{proof}

The crosscap number also features in Section~\ref{subsec:pretzel}.
We can improve this result slightly:

\begin{lemma}\label{lem:special}
Suppose $r(K) = 2$ and $\Delta_K(t) = 2-5t+2t^2$.  Then $g(K) = 1$.
\end{lemma}

\begin{proof}
In this case, there is a ribbon disk $\D$ for $K$ that corresponds to ribbon code $[(1,2,1,2)]$, and the surface $S$ constructed in the proof of Proposition~\ref{prop:cross} is orientable.
\end{proof}

Ribbon numbers for prime, non-alternating 12-crossing knots can be found in Table~\ref{table:12n}, and ribbon numbers for prime, alternating 12-crossing knots can be found in Table~\ref{table:12a}.
Determinants, Alexander polynomials, and knot genera included in the tables were obtained from KnotInfo~\cite{KnotInfo}.

\begin{table}[!ht]
    \centering
	\resizebox*{!}{.95\dimexpr\textheight-1\baselineskip\relax}{%
    \begin{tabular}{|l|l|l|l|l|l|l|}
    \hline
        $K$ & $\text{det}(K)$ & $\Delta_K$ & $g(K)$ & $r(K)$ & lower & upper \\ \hline
 $12n_4$ & 81 & $1-7t+19t^2-27t^3+19t^4-7t^5+t^6$ & 3 & 4 & Prop.~\ref{prop:r2r3} & Fig.~\ref{fig:ribbonA} \\ \hline
 $12n_{19}$ & 1 & $1-3t-t^2+7t^3-t^4-3t^5+t^6$ & 3 & 4 & Prop.~\ref{prop:r2r3} & Lem.~\ref{lem:symm} \\ \hline
 $12n_{23}$ & 9 & $2-5t+2t^2$ & 2 & 3, 4 & Lem.~\ref{lem:special} & Fig.~\ref{fig:ribbonA} \\ \hline
 $12n_{24}$ & 49 & $1-5t+11t^2-15t^3+11t^4-5t^5+t^6$ & 3 & 3 & Lem.~\ref{lem:genus} & Fig.~\ref{fig:ribbonA} \\ \hline
 $12n_{43}$ & 81 & $1-5t+10t^2-15t^3+19t^4-15t^5+10t^6-5t^7+t^8$ & 4 & 4 & Lem.~\ref{lem:genus} & Fig.~\ref{fig:ribbonA} \\ \hline
 $12n_{48}$ & 49 & $2-12t+21t^2-12t^3+2t^4$ & 2 & 4 & Prop.~\ref{prop:r2r3} & Lem.~\ref{lem:symm} \\ \hline
 $12n_{49}$ & 81 & $4-20t+33t^2-20t^3+4t^4$ & 2 & 4 & Prop.~\ref{prop:r2r3} & Fig.~\ref{fig:ribbonA} \\ \hline
 $12n_{51}$ & 9 & $2-5t+2t^2$ & 2 & 3 & Lem.~\ref{lem:special} & \cite[Fig. 5]{lamm} \\ \hline
 $12n_{56}$ & 9 & $1-2t+3t^2-2t^3+t^4$ & 3 & 3 & Lem.~\ref{lem:genus} & \cite[Fig. 5]{lamm} \\ \hline
 $12n_{57}$ & 9 & $1-2t+3t^2-2t^3+t^4$ & 2 & 3 & Prop.~\ref{prop:cross} & \cite[Fig. 5]{lamm} \\ \hline
 $12n_{62}$ & 81 & $2-9t+18t^2-23t^3+18t^4-9t^5+2t^6$ & 3 & 4 & Prop.~\ref{prop:r2r3} & \cite[Fig. 7]{lamm} \\ \hline
 $12n_{66}$ & 81 & $2-9t+18t^2-23t^3+18t^4-9t^5+2t^6$ & 3 & 4 & Prop.~\ref{prop:r2r3} & \cite[Fig. 7]{lamm} \\ \hline
 $12n_{87}$ & 49 & $4-12t+17t^2-12t^3+4t^4$ & 3 & 4 & Prop.~\ref{prop:r2r3} & Lem.~\ref{lem:symm} \\ \hline
 $12n_{106}$ & 81 & $1-4t+10t^2-16t^3+19t^4-16t^5+10t^6-4t^7+t^8$ & 4 & 4 & Lem.~\ref{lem:genus} & Fig.~\ref{fig:ribbonA} \\ \hline
 $12n_{145}$ & 25 & $1-6t+11t^2-6t^3+t^4$ & 2 & 3 & Prop.~\ref{prop:r2r3} & Lem.~\ref{lem:symm} \\ \hline
 $12n_{170}$ & 81 & $6-20t+29t^2-20t^3+6t^4$ & 2 & 4 & Prop.~\ref{prop:r2r3} & Fig.~\ref{fig:ribbonA} \\ \hline
 $12n_{214}$ & 1 & $1-t-t^2+3t^3-t^4-t^5+t^6$ & 3 & 3 & Lem.~\ref{lem:genus} & Lem.~\ref{lem:symm} \\ \hline
 $12n_{256}$ & 25 & $2-6t+9t^2-6t^3+2t^4$ & 3 & 3, 4 & Lem.~\ref{lem:genus} & Lem.~\ref{lem:symm} \\ \hline
 $12n_{257}$ & 25 & $2-6t+9t^2-6t^3+2t^4$ & 3 & 3 & Lem.~\ref{lem:genus} & Fig.~\ref{fig:ribbonA} \\ \hline
 $12n_{268}$ & 9 & $2-5t+2t^2$ & 2 & 3 & Lem.~\ref{lem:special} & \cite[Fig. 5]{lamm} \\ \hline
 $12n_{279}$ & 25 & $1-6t+11t^2-6t^3+t^4$  & 2 & 3 & Prop.~\ref{prop:r2r3} & Fig.~\ref{fig:ribbonA} \\ \hline
 $12n_{288}$ & 49 & $4-12t+17t^2-12t^3+4t^4$ & 2 & 4 & Prop.~\ref{prop:r2r3} & Lem.~\ref{lem:symm} \\ \hline
 $12n_{309}$ & 1 & $1-t-t^2+3t^3-t^4-t^5+t^6$ & 3 & 3 & Lem.~\ref{lem:genus} & Lem.~\ref{lem:symm} \\ \hline
 $12n_{312}$ & 49 & $1-5t+11t^2-15t^3+11t^4-5t^5+t^6$ & 3 & 3 & Lem.~\ref{lem:genus} & Fig.~\ref{fig:ribbonA} \\ \hline
 $12n_{313}$ & 1 & $1$ & 2 & 3 & Prop.~\ref{prop:cross} & Lem.~\ref{lem:symm} \\ \hline
 $12n_{318}$ & 1 & $1-t-t^2+3t^3-t^4-t^5+t^6$ & 3 & 3 & Lem.~\ref{lem:genus} & Lem.~\ref{lem:symm} \\ \hline
 $12n_{360}$ & 49 & $3-12t+19t^2-12t^3+3t^4$ & 2 & 3 & Prop.~\ref{prop:r2r3} & Fig.~\ref{fig:ribbonB} \\ \hline
 $12n_{380}$ & 81 & $2-9t+18t^2-23t^3+18t^4-9t^5+2t^6$ & 3 & 4 & Prop.~\ref{prop:r2r3} & Fig.~\ref{fig:ribbonB} \\ \hline
 $12n_{393}$ & 49 & $3-12t+19t^2-12t^3+3t^4$ & 2 & 3 & Prop.~\ref{prop:r2r3} & Fig.~\ref{fig:ribbonB} \\ \hline
 $12n_{394}$ & 25 & $1-6t+11t^2-6t^3+t^4$ & 2 & 3 & Prop.~\ref{prop:r2r3} & Fig.~\ref{fig:ribbonA} \\ \hline
 $12n_{397}$ & 49 & $1-5t+11t^2-15t^3+11t^4-5t^5+t^6$ & 3 & 3 & Lem.~\ref{lem:genus} & Fig.~\ref{fig:ribbonB} \\ \hline
 $12n_{399}$ & 81 & $1-7t+19t^2-27t^3+19t^4-7t^5+t^6$ & 3 & 4 & Prop.~\ref{prop:r2r3} & Fig.~\ref{fig:ribbonB} \\ \hline
 $12n_{414}$ & 25 & $2-6t+9t^2-6t^3+2t^4$ & 2 & 3 & Prop.~\ref{prop:r2r3} & Fig.~\ref{fig:ribbonB} \\ \hline
 $12n_{420}$ & 81 & $1-7t+19t^2-27t^3+19t^4-7t^5+t^6$ & 3 & 4 & Prop.~\ref{prop:r2r3} & Lem.~\ref{lem:symm} \\ \hline
 $12n_{430}$ & 1 & $1$ & 2 & 3 & Prop.~\ref{prop:cross} & Lem.~\ref{lem:symm} \\ \hline
 $12n_{440}$ & 81 & $2-9t+18t^2-23t^3+18t^4-9t^5+2t^6$ & 3 & 4 & Prop.~\ref{prop:r2r3} & \cite[Fig. 7]{lamm} \\ \hline
 $12n_{462}$ & 25 & $1-6t+11t^2-6t^3+t^4$ & 2 & 3 & Prop.~\ref{prop:r2r3} & Lem.~\ref{lem:symm} \\ \hline
 $12n_{501}$ & 49 & $4-12t+17t^2-12t^3+4t^4$ & 2 & 3 & Prop.~\ref{prop:r2r3} & Lem.~\ref{lem:symm} \\ \hline
 $12n_{504}$ & 121 & $1-6t+15t^2-24t^3+29t^4-24t^5+15t^6-6t^7+t^8$ & 4 & 4 & Lem.~\ref{lem:genus} &  Fig.~\ref{fig:ribbonB} \\ \hline
 $12n_{553}$ & 81 & $4-20t+33t^2-20t^3+4t^4$ & 2 & 4 & Prop.~\ref{prop:r2r3} &  Fig.~\ref{fig:ribbonB} \\ \hline
 $12n_{556}$ & 81 & $4-20t+33t^2-20t^3+4t^4$ & 2 & 4 & Prop.~\ref{prop:r2r3} &  Fig.~\ref{fig:ribbonB} \\ \hline
 $12n_{582}$ & 9 & $1-2t+3t^2-2t^3+t^4$ & 2 & 2 & Lem.~\ref{lem:genus} &  Lem.~\ref{lem:symm} \\ \hline
 $12n_{605}$ & 9 & $1-2t+4t^3-7t^4+4t^5-2t^7+t^8$ & 4 & 4 & Lem.~\ref{lem:genus} &  Lem.~\ref{lem:symm} \\ \hline
 $12n_{636}$ & 81 & $1-7t+19t^2-27t^3+19t^4-7t^5+t^6$ & 3 & 4 & Prop.~\ref{prop:r2r3} & Fig.~\ref{fig:ribbonB}  \\ \hline
 $12n_{657}$ & 81 & $1-4t+9t^2-16t^3+21t^4-16t^5+9t^6-4t^7+t^8$ & 4 & 4 & Lem.~\ref{lem:genus} & Fig.~\ref{fig:ribbonB}  \\ \hline
 $12n_{670}$ & 25 & $1-2t+3t^2-4t^3+5t^4-4t^5+3t^6-2t^7+t^8$ & 4 & 4 & Lem.~\ref{lem:genus} & Lem.~\ref{lem:symm}  \\ \hline
 $12n_{676}$ & 9 & $2-2t-4t^2+9t^3-4t^4-2t^5+2t^6$ & 3 & 4 & Prop.~\ref{prop:r2r3} & Fig.~\ref{fig:ribbonB}  \\ \hline
 $12n_{702}$ & 121 & $2-12t+28t^2-37t^3+28t^4-12t^5+2t^6$ & 3 & 4, 5 & Prop.~\ref{prop:r2r3} & Lem.~\ref{lem:symm}  \\ \hline
 $12n_{706}$ & 49 & $1-4t+6t^2-8t^3+11t^4-8t^5+6t^6-4t^7+t^8$ & 4 & 4 & Lem.~\ref{lem:genus} & Fig.~\ref{fig:ribbonB}  \\ \hline
 $12n_{708}$ & 49 & $1-3t+6t^2-9t^3+11t^4-9t^5+6t^6-3t^7+t^8$ & 4 & 4 & Lem.~\ref{lem:genus} & Lem.~\ref{lem:symm} \\ \hline
 $12n_{721}$ & 25 & $1-2t+3t^2-4t^3+5t^4-4t^5+3t^6-2t^7+t^8$ & 4 & 4 & Lem.~\ref{lem:genus} & Lem.~\ref{lem:symm} \\ \hline
 $12n_{768}$ & 25 & $1-3t+5t^2-7t^3+5t^4-3t^5+t^6$ & 3 & 3 & Lem.~\ref{lem:genus} & Lem.~\ref{lem:symm} \\ \hline
 $12n_{782}$ & 81 & $2-8t+18t^2-25t^3+18t^4-8t^5+2t^6$ & 3 & 4, 5 & Prop.~\ref{prop:r2r3} & Lem.~\ref{lem:symm} \\ \hline
 $12n_{802}$ & 121 & $1-5t+14t^2-25t^3+31t^4-25t^5+14t^6-5t^7+t^8$ & 4 & 4 & Lem.~\ref{lem:genus} & Fig.~\ref{fig:ribbonB} \\ \hline
 $12n_{817}$ & 49 & $2-6t+10t^2-13t^3+10t^4-6t^5+2t^6$ & 3 & 4 & Prop.~\ref{prop:r2r3} & Lem.~\ref{lem:symm} \\ \hline
 $12n_{838}$ & 25 & $1-6t+11t^2-6t^3+t^4$ & 2 & 3 & Prop.~\ref{prop:r2r3} & Lem.~\ref{lem:symm} \\ \hline
 $12n_{870}$ & 25 & $1-3t+5t^2-7t^3+5t^4-3t^5+t^6$ & 3 & 3 & Lem.~\ref{lem:genus} & Fig.~\ref{fig:ribbonA} \\ \hline
 $12n_{876}$ & 81 & $2-8t+18t^2-25t^3+18t^4-8t^5+2t^6$ & 3 & 4 & Prop.~\ref{prop:R4} & Fig.~\ref{fig:ribbonB} \\ \hline

    \end{tabular}}
	\caption{Ribbon numbers of prime, non-alternating 12-crossing knots}
	\label{table:12n}
\end{table}

\begin{table}[!ht]
    \centering
	\resizebox*{!}{.95\dimexpr\textheight-1\baselineskip\relax}{%
    \begin{tabular}{|l|l|l|l|l|l|l|}
    \hline
        $K$ & $\text{det}(K)$ & $\Delta_K$ & $g(K)$ & $r(K)$ & lower & upper \\ \hline
 $12a_3$ & 169 & $2-14t+40t^2-57t^3+40t^4-14t^5+2t^6$ & 3 & $[4,6]$ & Prop.~\ref{prop:r2r3} & Lem.~\ref{lem:symm} \\ \hline
 $12a_{54}$ & 169 & $3-17t+39t^2-51t^3+39t^4-17t^5+3t^6)$ & 3 & $[4,5]$ & Prop.~\ref{prop:r2r3} & Lem.~\ref{lem:symm} \\ \hline
 $12a_{77}$ & 225 & $1-7t+24t^2-49t^3+63t^4-49t^5+24t^6-7t^7+t^8$ & 4 & $[4,7]$ & Lem.~\ref{lem:genus} & Lem.~\ref{lem:symm} \\ \hline
 $12a_{100}$ & 225 & $3-21t+53t^2-71t^3+53t^4-21t^5+3t^6$ & 3 & $[4,7]$ & Prop.~\ref{prop:r2r3} & Lem.~\ref{lem:symm} \\ \hline
 $12a_{173}$ & 169 & $1-7t+20t^2-35t^3+43t^4-35t^5+20t^6-7t^7+t^8$ & 4 & $[4,6]$ & Lem.~\ref{lem:genus} & Lem.~\ref{lem:symm} \\ \hline
 $12a_{183}$ & 121 & $6-30t+49t^2-30t^3+6t^4$ & 2 & 5 & Prop.~\ref{prop:R4} & Lem.~\ref{lem:symm} \\ \hline
 $12a_{189}$ & 225 & $1-8t+26t^2-48t^3+59t^4-48t^5+26t^6-8t^7+t^8$ & 4 & $[4,7]$ & Lem.~\ref{lem:genus} & Lem.~\ref{lem:symm} \\ \hline
 $12a_{211}$ & 169 & $2-8t+20t^2-34t^3+41t^4-34t^5+20t^6-8t^7+2*t^8$ & 4 & 5 & Prop.~\ref{prop:R4} & Lem.~\ref{lem:symm} \\ \hline
 $12a_{221}$ & 169 & $2-14t+40t^2-57t^3+40t^4-14t^5+2t^6$ & 3 & $[4,5]$ & Prop.~\ref{prop:r2r3} & Lem.~\ref{lem:symm} \\ \hline
 $12a_{245}$ & 225 & $1-7t+24t^2-49t^3+63t^4-49t^5+24t^6-7t^7+t^8$ & 4 & $[4,7]$ & Prop.~\ref{prop:r2r3} & Lem.~\ref{lem:symm} \\ \hline
 $12a_{258}$ & 169 & $1-7t+20t^2-35t^3+43t^4-35t^5+20t^6-7t^7+t^8$ & 4 & $[4,5]$ & Lem.~\ref{lem:genus} & Lem.~\ref{lem:symm} \\ \hline
 $12a_{279}$ & 169 & $2-15t+40t^2-55t^3+40t^4-15t^5+2t^6$  & 3 & $[5,6]$ & Prop.~\ref{prop:R4} & Lem.~\ref{lem:symm} \\ \hline
 $12a_{348}$ & 225 & $2-17t+54t^2-79t^3+54t^4-17t^5+2t^6$ & 3 & 5 & Prop.~\ref{prop:R4} &  \cite[Fig. 8]{lamm} \\ \hline
 $12a_{377}$ & 225 & $1-8t+26t^2-48t^3+59t^4-48t^5+26t^6-8t^7+t^8$ & 4 & $[4,7]$ & Lem.~\ref{lem:genus} & Lem.~\ref{lem:symm} \\ \hline
 $12a_{425}$ & 81 & $6-20t+29t^2-20t^3+6t^4$ & 2 & $[4,5]$ & Prop.~\ref{prop:r2r3} & Lem.~\ref{lem:symm} \\ \hline
 $12a_{427}$ & 225 & $1-8t+26t^2-48t^3+59t^4-48t^5+26t^6-8t^7+t^8$ & 4 & 4 & Lem.~\ref{lem:genus} & Fig.~\ref{fig:ribbonA} \\ \hline
 $12a_{435}$ & 225 & $1-8t+26t^2-48t^3+59t^4-48t^5+26t^6-8t^7+t^8$ & 4 & 4 & Lem.~\ref{lem:genus} & Fig.~\ref{fig:ribbonA} \\ \hline
 $12a_{447}$ & 121 & $2-12t+28t^2-37t^3+28t^4-12t^5+2t^6$ & 3 & 4 & Prop.~\ref{prop:r2r3} & Lem.~\ref{lem:symm} \\ \hline
 $12a_{456}$ & 225 & $1-8t+25t^2-48t^3+61t^4-48t^5+25t^6-8t^7+t^8$ & 4 & $[5,7]$ & Prop.~\ref{prop:R4} & Lem.~\ref{lem:symm} \\ \hline
 $12a_{458}$ & 289 & $1-9t+32t^2-63t^3+79t^4-63t^5+32t^6-9t^7+t^8$ & 4 & $[5,7]$ & Prop.~\ref{prop:R4} & Lem.~\ref{lem:symm} \\ \hline
 $12a_{464}$ & 225 & $1-7t+24t^2-49t^3+63t^4-49t^5+24t^6-7t^7+t^8$ & 4 & 4 & Lem.~\ref{lem:genus} & Fig.~\ref{fig:ribbonA} \\ \hline
 $12a_{473}$ & 289 & $1-8t+30t^2-64t^3+83t^4-64t^5+30t^6-8t^7+t^8$ & 4 & $[5,7]$ & Prop.~\ref{prop:R4} & Lem.~\ref{lem:symm} \\ \hline
 $12a_{477}$ & 169 & $1-11t+41t^2-63t^3+41t^4-11t^5+t^6$ & 3 & $5$ & Prop.~\ref{prop:R4} & Lem.~\ref{lem:symm} \\ \hline
 $12a_{484}$ & 289 & $1-9t+32t^2-63t^3+79t^4-63t^5+32t^6-9t^7+t^8$ & 4 & $[5,6]$ & Prop.~\ref{prop:R4} & Lem.~\ref{lem:symm} \\ \hline
  $12a_{606}$ & 169 & $4-18t+38t^2-49t^3+38t^4-18t^5+4t^6$ & 3 & 5 & Prop.~\ref{prop:R4} & Lem.~\ref{lem:symm} \\ \hline
 $12a_{631}$ & 225 & $4-22t+52t^2-69t^3+52t^4-22t^5+4t^6$ & 4 & $[4,6]$ & Lem.~\ref{lem:genus} & Lem.~\ref{lem:symm} \\ \hline
 $12a_{646}$ & 169 & $2-9t+21t^2-33t^3+39t^4-33t^5+21t^6-9t^7+2t^8$ & 4 & $[5,6]$ & Prop.~\ref{prop:R4} & Lem.~\ref{lem:symm} \\ \hline
 $12a_{667}$ & 121 & $2-7t+15t^2-23t^3+27t^4-23t^5+15t^6-7t^7+2t^8$ & 4 & 5 & Prop.~\ref{prop:R4} & Lem.~\ref{lem:symm} \\ \hline
 $12a_{715}$ & 169 & $4-18t+38t^2-49t^3+38t^4-18t^5+4t^6$ & 3 & $[5,6]$ & Prop.~\ref{prop:R4} & Lem.~\ref{lem:symm} \\ \hline
 $12a_{786}$ & 169 & $2-15t+40t^2-55t^3+40t^4-15t^5+2t^6$ & 3 & $[5,6]$ & Prop.~\ref{prop:R4} & Lem.~\ref{lem:symm} \\ \hline
 $12a_{819}$ & 169 & $1-5t+12t^2-21t^3+29t^4-33t^5+29t^6-21t^7+12t^8-5t^9+t^{10}$ & 5 & 5 & Lem.~\ref{lem:genus} & Lem.~\ref{lem:symm} \\ \hline
 $12a_{879}$ & 121 & $2-7t+15t^2-23t^3+27t^4-23t^5+15t^6-7t^7+2t^8$ & 4 & 5 & Prop.~\ref{prop:R4} & Lem.~\ref{lem:symm} \\ \hline
 $12a_{887}$ & 289 & $1-9t+32t^2-63t^3+79t^4-63t^5+32t^6-9t^7+t^8$ & 4 & $[5,7]$ & Prop.~\ref{prop:R4} & Lem.~\ref{lem:symm} \\ \hline
 $12a_{975}$ & 225 & $4-22t+52t^2-69t^3+52t^4-22t^5+4t^6$ & 3 & 4 & Prop.~\ref{prop:r2r3} & Fig.~\ref{fig:ribbonA} \\ \hline
 $12a_{979}$ & 225 & $2-10t+27t^2-46t^3+55t^4-46t^5+27t^6-10t^7+2t^8$ & 4 & $[5,7]$ & Prop.~\ref{prop:R4} & Lem.~\ref{lem:symm} \\ \hline
 $12a_{990}$ & 225 & $1-8t+26t^2-48t^3+59t^4-48t^5+26t^6-8t^7+t^8$ & 4 & $[4,7]$ & Prop.~\ref{prop:r2r3} & \cite[Fig. 8]{lamm} \\ \hline
 $12a_{1011}$ & 121 & $1-4t+9t^2-15t^3+20t^4-23t^5+20t^6-15t^7+9t^8-4t^9+t^{10}$ & 5 & 5 & Lem.~\ref{lem:genus} &Lem.~\ref{lem:symm} \\ \hline
 $12a_{1019}$ & 361 & $1-10t+39t^2-80t^3+101t^4-80t^5+39t^6-10t^7+t^8$ & 4 & $[5,6]$ & Prop.~\ref{prop:R4} &Lem.~\ref{lem:symm} \\ \hline
 $12a_{1029}$ & 81 & $2-6t+10t^2-14t^3+17t^4-14t^5+10t^6-6t^7+2t^8$ & 4 & 5 & Prop.~\ref{prop:R4} &Lem.~\ref{lem:symm} \\ \hline
 $12a_{1034}$ & 121 & $8-30t+45t^2-30t^3+8t^4$ & 2 & 5 & Prop.~\ref{prop:R4} & Lem.~\ref{lem:symm} \\ \hline
 $12a_{1083}$ & 169 & $2-9t+21t^2-33t^3+39t^4-33t^5+21t^6-9t^7+2t^8$ & 4 & $[5,6]$ & Prop.~\ref{prop:R4} &Lem.~\ref{lem:symm} \\ \hline
 $12a_{1087}$ & 225 & $1-8t+25t^2-48t^3+61t^4-48t^5+25t^6-8t^7+t^8$ & 4 & $[5,7]$ & Prop.~\ref{prop:R4} &Lem.~\ref{lem:symm} \\ \hline
 $12a_{1105}$ & 289 & $1-8t+30t^2-64t^3+83t^4-64t^5+30t^6-8t^7+t^8$ & 4 & $[5,6]$ & Prop.~\ref{prop:R4} &Lem.~\ref{lem:symm} \\ \hline
 $12a_{1119}$ & 169 & $2-9t+21t^2-33t^3+39t^4-33t^5+21t^6-9t^7+2t^8$ & 4 & 5 & Prop.~\ref{prop:R4} &Lem.~\ref{lem:symm} \\ \hline
 $12a_{1202}$ & 169 & $9-42t+67t^2-42t^3+9t^4$ & 2 & $[5,6]$ & Prop.~\ref{prop:R4} &Lem.~\ref{lem:symm} \\ \hline
 $12a_{1225}$ & 225 & $1-5t+14t^2-28t^3+41t^4-47t^5+41t^6-28t^7+14t^8-5t^9+t^{10}$ & 5 & $[5,7]$ & Lem.~\ref{lem:genus} & \cite[Fig. 8]{lamm} \\ \hline
 $12a_{1269}$ & 169 & $4-17t+38t^2-51t^3+38t^4-17t^5+4t^6$ & 3 & 5 & Prop.~\ref{prop:R4} & Lem.~\ref{lem:symm} \\ \hline
 $12a_{1277}$ & 121 & $4-14t+26t^2-33t^3+26t^4-14t^5+4t^6$ & 3 & 5 & Prop.~\ref{prop:R4} & Lem.~\ref{lem:symm} \\ \hline
 $12a_{1283}$ & 81 & $1-3t+6t^2-10t^3+13t^4-15t^5+13t^6-10t^7+6t^8-3t^9+t^{10}$ & 5 & 5 & Lem.~\ref{lem:genus} & Lem.~\ref{lem:symm} \\ \hline
\end{tabular}}
	\caption{Ribbon numbers of prime, alternating 12-crossing knots}
	\label{table:12a}
\end{table}

\section{Higher-genus ribbon numbers}\label{sec:hgrn}

In this section, we transition from the study of ribbon numbers of ribbon disks bounded by knots to the more general study of ribbon numbers of higher-genus ribbon surfaces bounded by knots.
The result is a more refined four-dimensional complexity measure for knots: the higher-genus ribbon number spectrum.
We apply the techniques introduced above to calculate this complexity measure in certain interesting cases, and we also discuss how some techniques that were useful in calculating ribbon numbers of knots fail to be useful in calculating higher-genus ribbon numbers.

A \emph{ribbon surface} for a knot $K$ is a surface $F\subset S^3$ with $\partial F = K$ such that $F$ is immersed with (only) finitely many ribbon intersections.
The \emph{ribbon number} $r(F)$ of $F$ is the number of ribbon intersections.
The \emph{genus-$g$ ribbon number} of a knot $K$ is the minimum
$$r_g(K) = \min\{r(F)\,|\, \text{ $F$ is an orientable, genus-$g$ ribbon surface for $K$}\},$$
where we adopt the convention that $r_g(K) = \infty$ if $K$ does not bound a ribbon surface of genus $g$.  To avoid confusion and for consistency of notation, we use $r_0(K)$ in this section to denote the ribbon number of $K$.
The \emph{higher-genus ribbon number spectrum}, or \emph{ribbon spectrum} for short, of a knot $K$ is the finite sequence
$$\frak r(K) = (r_0(K), r_1(K),r_2(K),\ldots, r_{g(K)}(K))$$
of higher-genus ribbon numbers of $K$, where $g(K)$ denotes the Seifert genus of $K$.

\begin{remark}
\label{rmk:spectrum}
	We note some immediate properties of the ribbon spectrum of a knot.
	\begin{enumerate}
		\item We have $r_g(K) = 0$ if and only if $g\geq g(K)$, so the spectrum truncates to a well-defined finite sequence whose length is $g(K) + 1$.
		\item By the convention that $r_g(K)=\infty$ if $K$ does not bound a ribbon surface of genus $g$, we have $r_g(K) = \infty$ whenever $g<g_4(K)$, where $g_4(K)$ denotes the smooth 4-genus of $K$.
		\item Given a genus-$g$ ribbon surface for $K$ with $r_g(K)$ ribbon intersections, we can apply the \emph{cellar door trick} depicted in Figure~\ref{fig:cellar_door} to one of the ribbon intersections to obtain a ribbon surface for $K$ with genus $r_g(K)-1$.  When a ribbon surface is orientable, one version of cellar door trick will result in an orientable surface, and the other will result in a non-orientable surface.  It follows that we always have $r_g(K) <  r_{g-1}(K)$ for all $g$, so the higher-genus ribbon number spectrum of a knot is decreasing (when finite).  (The term ``cellar door" refers to the image of the two door panels of a traditional outdoor cellar entrance; in this case, one panel opens ``outward", and the other opens ``inward.")

\begin{figure}[h!]
	\centering
	\includegraphics[width=.8\textwidth]{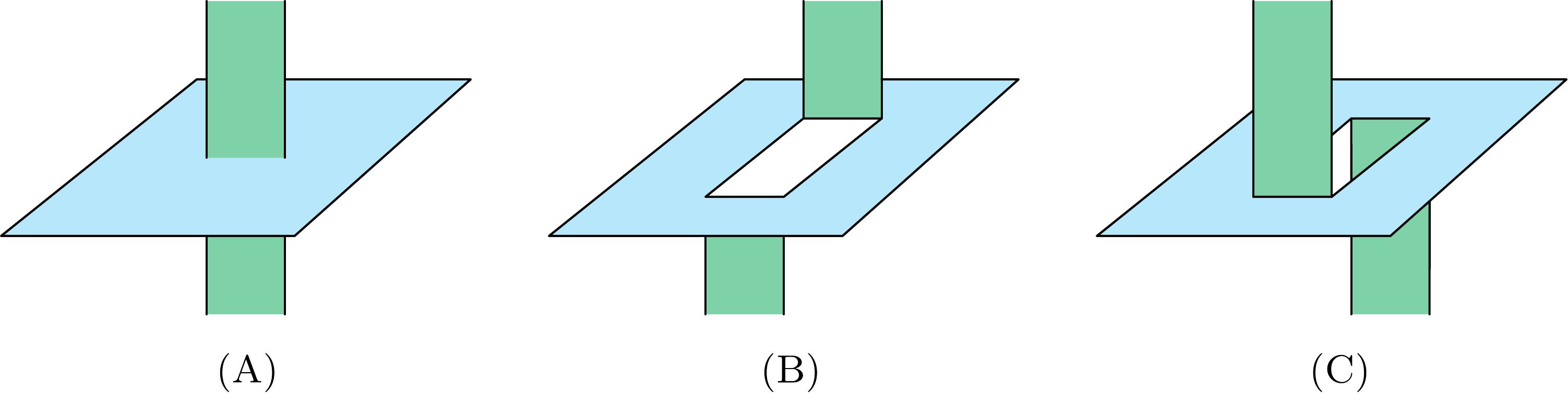}
	\caption{The cellar door trick can be used to turn a ribbon surface $F$ into a new ribbon surface $F'$ with one fewer ribbon singularity and $\chi(F') = \chi(F)-2$.  When $F$ is orientable, one version of the trick yields an orientable surface, and the other version yields a non-orientable surface.}
	\label{fig:cellar_door}
\end{figure}

\item We consider the higher-genus ribbon number spectrum to be an interesting indicator of the four-dimensional behavior of a knot.
For example, the spectrum is trivial in the case that $g(K)=g_4(K)$, reflecting the four-dimensional rigidity of such knots, which include, for example, strongly quasipositive knots~\cite{Shuma}.

\item In the case that $r_0(K) = g(K)$, the ribbon number spectrum of $K$ is also understood; by Proposition~\ref{prop:hgrn_bound} below, we have
\[ \frak r(K) = (r_0(K), r_0(K) - 1, r_0(K) - 2, \ldots , 1, 0).\]
We call such a ribbon number spectrum a \emph{stair-step} spectrum.
	\end{enumerate}
\end{remark}

\subsection{Higher-genus ribbon numbers of pretzel knots}
\label{subsec:pretzel}
\ 

The goal of this subsection is to give some examples of knots where the spectrum differs from the examples above and to explore the effectiveness and limitations of the techniques of the paper in bounding higher-genus ribbon numbers.
We pose the following as a motivating conjecture.

\begin{conjecture}
\label{conj:pretzel}
	Let $p,q\in\Z$ be odd, and let $n\in\N$.
	If $K$ is the $(2n+1)$-stranded pretzel knot
	$$K = P(q,p,-p,\ldots,p,-p)$$
	and $0\leq g\leq n$,
	then $r_g(K) = (n-g)(p-1)$, so
	$$\frak r(K) = (n(p-1), (n-1)(p-1), (n-2)(p-1),\ldots, p-1,0).$$
\end{conjecture}

In particular, this conjecture posits that the ribbon spectrum can have arbitrarily many jumps of arbitrarily large decrease.
We can give upper bounds consistent with this conjecture, but finding sharp lower bounds is more challenging.
Our main result in this direction is the following.

\begin{reptheorem}{thm:pretzel}
	For odd $q\in\Z$, the $5$--stranded pretzel knot $K = P(q,3,-3,3,-3)$ satisfies	
	$$\frak r(K) = (4,2,0).$$
\end{reptheorem}

Ribbon surfaces realizing these values are shown in Figure~\ref{fig:5-strand}.
We prove this theorem by establishing upper and lower bounds on $r_g(K)$ for the more general family of pretzel knots appearing in Conjecture~\ref{conj:pretzel}.
The upper bound is constructive.

\begin{figure}[h!]
	\centering
	\includegraphics[width=.9\textwidth]{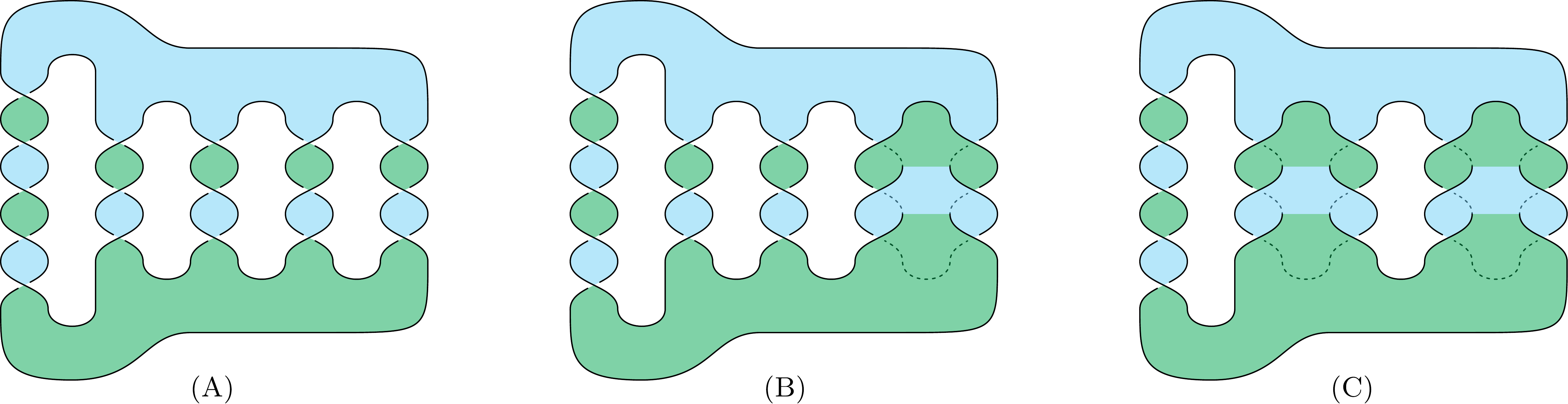}
	\caption{Ribbon surfaces for the pretzel knot $P(-5,3, -3, 3, -3)$}
	\label{fig:5-strand}
\end{figure}

\begin{lemma}
\label{lem:pretzel_upper}
	If $K$ is a pretzel knot satisfying the hypotheses of Conjecture~\ref{conj:pretzel}, and if $0~\leq~g~\leq~n$, then
	$$r_g(K) \leq (n-g)(p-1).$$
\end{lemma}

\begin{proof}
	Figure~\ref{fig:5-strand} shows ribbon surfaces with the desired number of ribbon intersections in the case of $n=2$ and $p=3$.  When $g=1$, the surface depicted has one ``clasp region," and when $g=0$, the surface depicted has two ``clasp regions."  In both cases, each clasp region contains two ribbon intersections.  In the general case, we can construct a genus-$g$ ribbon surface for $K$ with $n-g$ clasp regions, each of which contains $p-1$ ribbon intersections, for a total of $(n-g)(p-1)$ ribbon intersections.	
\end{proof}

Note that Gabai established that if $K$ is a pretzel knot satisfying the hypotheses of Conjecture~\ref{conj:pretzel}, then $g(K) = n$~\cite[Theorem~3.2]{gabai}, so that $r_g(K) = 0$ when $g \geq n$.  For the lower bounds, we make use of the following connection between the higher-genus ribbon numbers of a knot and its Seifert genus and crosscap number.

\begin{proposition}
\label{prop:hgrn_bound}
	Let $K$ be a knot with Seifert genus $g(K)$ and crosscap number $\gamma(K)$.
	For any $g\geq 0$, we have:
	\begin{enumerate}
		\item $r_g(K) \geq g(K) - g$; and
		\item If $r_g(K) > 0$, then $2 r_g(K) \geq \gamma(K) - 2g$.
	\end{enumerate}
\end{proposition}

\begin{proof}
	Let $F$ be a genus-$g$ ribbon surface for $K$ with ribbon number $r_g(K)$.
	We can apply the orientable cellar door trick to each ribbon intersection to get an embedded, orientable surface $F'$ of genus $g+r_g(K)$.
	(Each application of the trick increases the genus by one; see Figure~\ref{fig:cellar_door}.)
	We must have $g + r_g(K) \geq g(K)$, so inequality (1) follows.
	
	Alternatively, as long the genus-$g$ ribbon surface $F$ is not embedded, applying the cellar door trick to each ribbon intersection so that at least one application disrespects orientation yields a non-orientable surface $F'$.
	Each application of the trick contributes two to the crosscap number of the surface.
	In the end, we must have $2g + 2 r_g(K) \geq \gamma(K)$, which gives inequality (2).
\end{proof}

A key ingredient to our proof of Theorem~\ref{thm:pretzel} is work of Ichihara and Mizushima that determines the crosscap numbers of the pretzel knots currently under consideration~\cite{IchMiz}.
We use their result in the proof of our next lemma.

\begin{lemma}
\label{lem:pretzel_lower}
	Let $K$ be a pretzel knot satisfying the hypotheses of Conjecture~\ref{conj:pretzel}.
	For $0\leq g < n$, we have
$r_g(K) \geq n-g+1$.
\end{lemma}

\begin{proof}
	Ichihara and Mizushima have established that if $K$ is a $(2n+1)$-stranded pretzel knot with all odd twist parameters, then $\gamma(K)=2n+1$~\cite{IchMiz}.  As observed above, if $g < n$, then $r_g(K) > 0$, and so we can apply Proposition~\ref{prop:hgrn_bound}, yielding the inequality
	$$2 r_g(K) \geq 2n+1 - 2g.$$
	Since $r_g(K)$, $n$, and $g$ are integers, this implies
	$$r_g(K) \geq n-g + 1.$$
\end{proof}

Combining the above work, we prove the main result of this section.

\begin{proof}[Proof of Theorem~\ref{thm:pretzel}]
	Let $K=P(q,3,-3,3,-3)$.
	The ribbon disk shown in Figure~\ref{fig:5-strand} has ribbon code isomorphic to $([1,2,5],[2,1,5],[3,4,5],[4,3,5])$, and we can compute
	$$\Delta_K(t) = 4-20t+33t^2-20t^3+4t^4.$$
	Since $\Delta_K(t) \notin \Rr_3$, we have $r_0(K)\geq 4$.
	This, combined with existence of the ribbon disk in Figure~\ref{fig:5-strand}, shows that $r_0(K)= 4$.
	
	Applying Lemma~\ref{lem:pretzel_lower} to the present setting where $n=2$, we have $r_1(K)\geq 2$.
	This, combined with the the existence of the ribbon once-punctured torus exhibited in Figure~\ref{fig:5-strand}, shows that $ r_1(K)=2$.
	Finally, the Seifert surface shown in Figure~\ref{fig:5-strand} shows that $g(K)\leq 2$, so $\frak r_2(K)=0$, completing the proof.
\end{proof}

\begin{remark}
	The lower bound given in Lemma~\ref{lem:pretzel_lower} is independent of $p$, yet we expect the higher-genus ribbon numbers of $K=P(q,p,-p,\ldots, p,-p)$ to grow linearly with $p$.
	Thus, the lower bounds coming from $g(K)$ and $\gamma(K)$ are not sufficient to approach Conjecture~\ref{conj:pretzel}.
	As $n$ increases, the lower bounds become ineffective for $g\leq g(K)-2$.
	For example, when $p=3$ and $n=4$, we have the following bounds on the ribbon spectrum of $K$:
	$$(5,4,3,2,0) \leq \frak r(K)\leq (8,6,4,2,0).$$
\end{remark}

\begin{remark}
	There is at least one additional knot in the table whose ribbon spectrum can be seen to have two non-trivial ``jumps" using the present techniques.
	Let $K = \mathbf{12a_{1202}}$.
	From Table~\ref{table:12a}, we know that $r_0(K) \geq 5$.
	KnotInfo~\cite{KnotInfo} asserts that $\gamma(K)=5$.
	It follows from Proposition~\ref{prop:hgrn_bound} that $r_1(K)\geq 2$.
	Figure~\ref{fig:12a1202_torus} shows that $r_1(K)\leq 3$, so we have
	$$(5,2,0) \leq \frak r(K)\leq (6,3,0).$$
	We conclude that $K$ has two jumps in its higher-genus ribbon number spectrum.
	It is worth remarking that $K$ is one of only three slice knots with 12 or fewer crossings for which it is unknown whether it is doubly slice~\cite{LivMei}.
\end{remark}

\begin{figure}[h!]
	\centering
	\includegraphics[width=.8\textwidth]{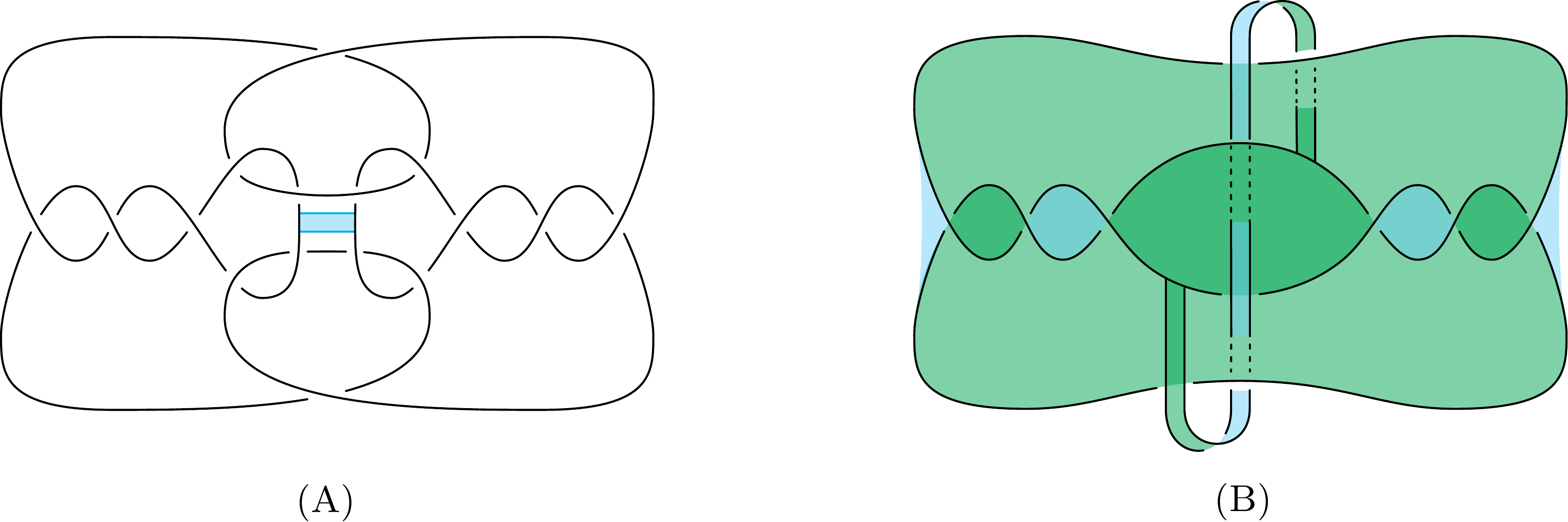}
	\caption{The knot $\mathbf{12_{a1202}}$ shown in (A) with a ribbon band and shown in (B) bounding a genus-one ribbon surface with ribbon number three}
	\label{fig:12a1202_torus}
\end{figure}

\subsection{Failure of obstruction for higher-genus ribbon numbers}
\ 

Given our work on (genus-0) ribbon numbers at the beginning of the paper, a natural question is

\begin{question}
Can Alexander polynomials be used to exhibit novel lower bounds on $r_1(K)$?
\end{question}

In this subsection, we prove that the answer is no.  As noted in the introduction, if the degree of $\Delta_K(t)$ is $2d$, then $g(K) \geq d$, and using Proposition~\ref{prop:hgrn_bound}, we can show $g_1(K) \geq g(K) - 1 \geq d-1$.  We will demonstrate that we can obtain no further data about $r_1(K)$ from $\Delta_K(t)$ beyond that which is filtered through $g(K)$.

\begin{reptheorem}{thm:torus_alex}
	If $K$ is a ribbon knot such that $\text{deg}(\Delta_K(t))=2d$, then there exists a knot $K'$ such that $\Delta_{K'} (t)=\Delta_K(t)$ and $r_1(K') = d-1$. 
\end{reptheorem}

In particular, the collection of Alexander polynomials of all knots with bounded $r_1$ is infinite, and their determinants are unbounded, in contrast to the behavior of these invariants with respect to $r_0$.  To prove the proposition, we will first define a family of ribbon codes that give rise to a suitably broad class of Alexander polynomials.

Let $m \geq 1$ and let $(n_0,\dots,n_m)$ be an $(m+1)$-tuple of integers.  Define the ribbon code
\[ \Gamma(n_0,\dots,n_m) = ([1,\underbrace{a_0,b_0,\dots,a_0,b_0}_{2|n_0|},-1,\underbrace{a_1,b_1,\dots,a_1,b_1}_{2|n_1|},-1,\dots,-1,\underbrace{a_m,b_m,\dots,a_m,b_m}_{2|n_m|},2]),\]
where $(a_i,b_i) = (-1,-2)$ if $n_i$ is positive, $(a_i,b_i) = (2,1)$ if $n_i$ is negative, and $a_i$ and $b_i$ do not appear in the code if $n_i = 0$.
The notation
\[ \underbrace{a_0,b_0,\dots,a_0,b_0}_{2|n_0|}\]
means that  $a_0$ and $b_0$ alternately $|n_0|$ times each, so that $F(\Gamma(n_0,\dots,n_m)) = 1$ and $r(\Gamma(n_0,\dots,n_m)) = m + 2 \sum |n_i|$.  This ribbon code is integral to our construction of the knots in Theorem~\ref{thm:torus_alex}, and we compute its Alexander polynomial below using Fox calculus, which we now recall.

\subsection{Disk groups and Fox calculus}
\label{subsec:groups}

For any immersed ribbon disk $\D$ in $S^3$, we can perturb $\D$ near its ribbon self-intersections to get an embedded disk in $B^4$, which we also denote $\D$, abusing notation.  We first recall how to obtain a presentation for the fundamental group $\pi(\D) = \pi_1(B^4\setminus\D)$ from a ribbon code for $\Dd$; this is a well-known process~\cite{howie_asphericity,yasuda:2-knot}.
Given a disk-band presentation $(D,B)$ for $\Dd$ with corresponding ribbon code $\Gamma$, recall that the exterior $B^4\setminus\nu(\Dd)$ can be built by taking one 1--handle for each disk in $D$, together with one 2--handle for each band in $B$~\cite[Chapter~6.2]{gompf-stipsicz}.
This gives a presentation for $\pi(\D)$ with a generator for each disk in $D$ (hence each vertex of $\Gamma$) and a relation for each band in $B$ (hence each marked edge of $\Gamma$)~\cite[Section~4.2]{gompf-stipsicz}.

Letting $x_i$ be the generator corresponding to the vertex $v_i$ of $\Gamma$, an edge $e_j$ of the form $[i_1,\ell_1,\ldots, \ell_m,i_2]$ corresponds to the relation $r_j = x_{i_1}w_j x_{i_2}^{-1}w_j^{-1}$, where $w_j = x_{|\ell_1|}^{\epsilon_1}\cdots x_{|\ell_m|}^{\epsilon_m}$.
Each instance $x_{|\ell_k|}^{\epsilon_k}$ in $w_j$ corresponds to a ribbon intersection $\alpha_k$ between the band $\frak b_j$ and the disk $D_{|\ell_k|}$.
Directing $e_j$ from $v_{i_1}$ to $v_{i_2}$ gives a direction for the band $\frak b_j$.
First note that $\epsilon_k$ will be positive (respectively, negative) if this band direction agrees with (respectively, disagrees with) the normal direction of $D_{|\ell_k|}$ at $\alpha_k$.
In terms of the ribbon code, this means that $\epsilon_k = \text{sign}(\ell_k)$ (respectively, $\epsilon_k = -\text{sign}(\ell_k)$) if the direction of $e_j$ agrees with (respectively, disagrees with) the local direction at $\alpha_k$.
For example, the ribbon code from Example~\ref{ex:code} (corresponding to the ribbon disk shown in Figure~\ref{fig:code}) gives rise to the presentation
$$\langle x_1,\, x_2,\,  x_3\,|\, x_1x_3^{-1}x_2x_3^{-1}x_2^{-1}x_3x_2^{-1}x_3,\  x_2x_1^{-1}x_3^{-1}x_1\rangle.$$

Let $\frac{\partial r_j}{\partial x_i}$ denote the Fox derivative of the relation $r_j$ with respect to the generator $x_i$~\cite{crowell-fox}.
Then, the matrix $A = (a_{ji})$, where $a_{ji} = \frac{\partial r_j}{\partial x_i}\vert_{x_i\mapsto t}$ is the Alexander matrix of $\Dd$:
It is a presentation matrix for the Alexander module, i.e., the first homology of the infinite cyclic cover of $\Dd$ (equivalently, of the 2-knot obtained by doubling $\Dd$), considered as a $\Z[t,t^{-1}]$-module.
Note that this matrix has a row for each relation and a column for each generator.
Letting $n$ denote the number of generators, the ideal generated by the $(n-1)$--minors of $A$ (i.e., the first elementary ideal) is called the Alexander ideal of the disk, and the characteristic polynomial $\Delta_\Dd(t)$ of this ideal is called the Alexander polynomial of the disk~\cite[Chapter~7]{kawauchi_survey}.
For example, the ribbon disk from Example~\ref{ex:code} and Figure~\ref{fig:code} has Alexander matrix
$$ A = \begin{bmatrix}
	1 & 1-2t^{-1} & -2+2t^{-1} \\
	-1+t^{-1} & 1 & -t^{-1}
\end{bmatrix},$$
so the Alexander module is cyclically generated by $\Delta_\Dd(t) = 2 - 3t + 2t^2$.

We are now prepared to calculate the Alexander polynomial of a boundary knot for the ribbon code $\Gamma(n_0,\ldots,n_m)$ described above.

\begin{lemma}\label{lem:alexander}
	Let $\Gamma = \Gamma(n_0,\dots,n_m)$ be defined as above.  We have $\Delta_\Gamma(t)~=~f(t)f(t^{-1})$, where
	$$f(t)=-n_0 +\sum_{i=1}^m(n_{i-1}-n_i)t^i + n_mt^{m+1} + t^m .$$
\end{lemma}

\begin{proof}
Let $\D$ be a ribbon disk giving rise to $\Gamma$.  As above, we can perturb $\D$ to get an embedded disk in $B^4$, and by Lemma 3.1 of~\cite{FMZ}, we have
\[\Delta_\Gamma(t) = \Delta_{\partial \D}(t) = \Delta_{\D}(t) \cdot \Delta_{\D}(t^{-1}).\]
Thus, it suffices to show that $\Delta_{\D}(t) = f(t)$. 
According to the procedure given above, the fundamental group $\pi(\D)$ of the exterior of this embedded ribbon disk can be computed from the ribbon code to be presented as
$$\pi(\D) = \langle x_1, x_2 \mid x_1 w x_2^{-1}w^{-1}\rangle,$$
where
$$w=(x_1x_2^{-1})^{n_0} \prod_{i=1}^m x_1 (x_1x_2^{-1})^{n_i}.$$

To make the Fox calculus derivations more efficient, we let
	$$\left.\frac{\partial}{\partial x_1}\right\vert_t z = \left.\frac{\partial}{\partial x_1} z\right\vert_{x_1,x_2\mapsto t}$$
for any word $z\in\pi(\D)$.
In other words, we evaluate $x_1\mapsto t$ and $x_2\mapsto t$ as we derive.
We first calculate the evaluated Fox derivatives for $w$ and $w^{-1}$.
We have
\begin{eqnarray*}
	\left.\frac{\partial}{\partial x_1}\right\vert_t w & = & \left.\frac{\partial}{\partial x_1}\right\vert_t (x_1x_2^{-1})^{n_0} +\sum_{i=1}^m t^{i-1}\left.\frac{\partial}{\partial x_1}\right\vert_t x_1(x_1x_2^{-1})^{n_i} \\
	& = & n_0+\sum_{i=1}^m t^{i-1}(1+n_i t),
\end{eqnarray*}
while
\begin{eqnarray*}
	\left.\frac{\partial}{\partial x_1}\right\vert_t w^{-1} & = & \sum_{i=0}^{m-1} t^{-i}\left.\frac{\partial}{\partial x_1}\right\vert_t (x_2x_1^{-1})^{n_{m-i}}x_1^{-1} + t^{-m}\left.\frac{\partial}{\partial x_1}\right\vert_t (x_2x_1^{-1})^{n_0}\\
	& = & \sum_{i=0}^{m-1}t^{-i}(-n_{m-i}-t^{-1}) - n_0t^{-m}  \\
	& = & -n_0t^{-m}-\sum_{i=0}^{m-1} t^{-m+i+1}(n_{i+1}+t^{-1}),
\end{eqnarray*}
where we have re-indexed the sum between the last two lines.

From this, we calculate
\begin{eqnarray*}
	\left.\frac{\partial}{\partial x_1}\right\vert_t x_1 w x_2^{-1}w^{-1} & = & 1 + t\left(n_0+\sum_{i=1}^mt^{i-1}(1+n_it)\right) \\
	&& \hspace{2cm} + t^m\left(-n_0t^{-m} - \sum_{i=0}^{m-1}t^{-m+i+1}(n_{i+1}+t^{-1})\right) \\
	& = & \sum_{i=0}^mt^i(1+n_it) - n_0 - \sum_{i=0}^{m-1}t^{i+1}(n_{i+1} + t^{-1}) \\
	& = & \sum_{i=0}^{m-1}(t^i + n_it^{i+1} - n_{i+1}t^{i+1}-t^i) - n_0 + t^m + n_mt^{m+1} \\
	& = & -n_0+\sum_{i=1}^m(n_{i-1}-n_i)t^i + n_mt^{m+1} + t^m,
\end{eqnarray*}
where we have re-indexed the sum between the last two lines.

The computation for the Fox derivative $\left.\frac{\partial}{\partial x_2}\right\vert_t x_1 w x_2^{-1}w^{-1}$ is similar and yields the same polynomial, but negated.
It follows that the Alexander matrix is $\begin{bmatrix}f(t) & -f(t)\end{bmatrix}$, so the Alexander polynomial is $\Delta_\D(t) = f(t)$.
As desired, we deduce $\Delta_\Gamma(t) = f(t)f(t^{-1})$.
\end{proof}

Finally, given the ribbon code $\Gamma(n_0,\dots,n_m)$, we construct a corresponding ribbon knot $K(n_0,\dots,n_m)$, which we will use to prove Theorem~\ref{thm:torus_alex}.

\begin{definition}
\label{def:ribbon_torus}
	As above, let $(n_0, \ldots, n_m)$ be a sequence of integers with $m\geq 1$.
	Let $A$ be an annulus in $S^3$ whose boundary is a 2-component unlink, and let $D_1$ and $D_2$ be disjoint disks bounded by $\partial A$, so that $D_1 \cup A \cup D_2$ is an embedded 2-sphere $S$. 
	Orient these three surfaces so that the positive normal directions for the $D_i$ point into $S$, while the positive normal direction for $A$ points out of $S$; see Figure~\ref{fig:ribbon_torus}, where the normal directions point out of the green side of each surface.
	Create a band $\frak b$ in the following manner; see Figure~\ref{fig:ribbon_torus}(A), where the case of $(4,-2)$ is shown.
    \begin{enumerate}
        \item Attach one end of the band $\frak b$ to $\partial D_1$, so that $\frak b$ begins on the outside of $S$.
        \item \begin{enumerate}
        \item If $n_0 > 0$, wind $\frak b$ around $A$ a total of $n_0$ times by passing first through $D_1$ and then through $D_2$.
        \item If $n_0 < 0$, wind $\frak b$ around $A$ a total of $|n_0|$ times by passing first through $D_2$ and then through $D_1$.
        \item If $n_0 = 0$, do not wind $\frak b$ around $A$ at all.
        \end{enumerate}
  \item Do the following for each $i\in\{1,\ldots,m\}$: 
        \begin{enumerate}
            \item Pass $\frak b$ through $D_1$ once (crossing $S$ from the outside to the inside) and then through $A$ once (passing back to the outside of $S$).
            \item Wind $\frak b$ around the annulus $|n_i|$ times with the same conventions as in step (2).
        \end{enumerate}
        \item Attach the other end of $\frak b$ to $\partial D_2$ so that $A \cup \frak b$ is an orientable ribbon surface.
    \end{enumerate}

Let $F=F(n_0, \ldots, n_m)$ denote the ribbon surface $A \cup \frak b$, which is a once-punctured torus with $m$ ribbon intersections, and let $K = K(n_0,\dots,n_m) = \pd F$.  By construction, $\D = \D(n_0,\dots,n_m) = D_1 \cup \frak b \cup D_2$ is a ribbon disk for $K$ with $m + 2 \sum|a_i|$ ribbon intersections.
Note that neither $F$ nor $\D$ is uniquely defined up to isotopy, but regardless of any choices made in the construction above, $\D$ gives rise to the ribbon code $\Gamma(n_0,\dots,n_m)$.
\end{definition}

\begin{figure}[h!]
	\centering
	\includegraphics[width=.9\textwidth]{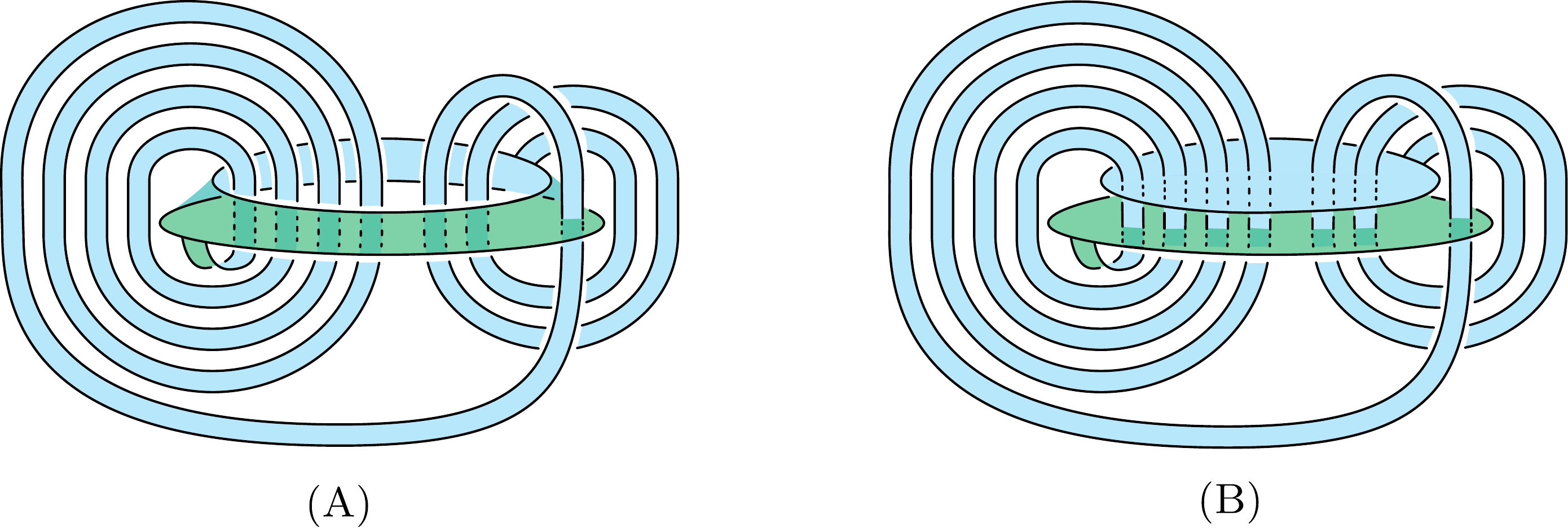}
	\caption{(A) The genus-one ribbon surface $F(4,-2)$, and (B) the ribbon disk $\D(4,-2)$.}
	\label{fig:ribbon_torus}
\end{figure}

To conclude, we use $K(n_0,\dots,n_m)$ to prove Theorem~\ref{thm:torus_alex}.

\begin{proof}[Proof of Theorem~\ref{thm:torus_alex}]
	Since $K$ is ribbon, we have $\Delta_K(t)=g(t)g(t^{-1})$ for some polynomial $g(t)$ with $g(1)=1$.
	Let $m=d-1$, and write
	$$g(t)=\sum_{i=0}^{m+1} a_i t^i.$$
	For $0\leq j\leq m-1$, define
	$$n_j=-\sum_{i=0}^j a_i,$$
	and define
	$$n_m = 1 -\sum_{i=0}^ma_i.$$
	Note that this implies the following:
	\begin{enumerate}
		\item $n_0=-a_0$,
		\item $n_{j-1}-n_j = a_j$ for $1\leq j\leq m-1$,
		\item $n_{m-1}-n_m = a_m-1$, and
		\item $a_{m+1} = 1 - \sum_{i=0}^ma_i = n_m$,
	\end{enumerate}
	with the last fact following from the fact that $g(1)=1$.
		
	Now, let $F=F(n_0, \dots, n_m)$, and let $K' = \partial F$.  Then $K'$ bounds the ribbon disk $\D(n_0,\dots,n_m)$ with corresponding ribbon code $\Gamma(n_0,\dots,n_m)$.
	From Lemma~\ref{lem:alexander}, we have
	$$\Delta_{K'} (t)=f(t)f(t^{-1}),$$
	where
	$$f(t)  =  -n_0 + \sum_{i=1}^m(n_{i-1}-n_i)t^i +  n_mt^{m+1} +t^m.$$
	Using the above facts, we have
	\begin{eqnarray*}
		f(t) & = & a_0 + \sum_{i=1}^{m-1}a_it^i + (n_{m-1}-n_m)t^m + a_{m+1}t^{m+1} +  t^m \\
		& = & \sum_{i=0}^{m-1} a_i t^i + a_{m+1}t^{m+1} + (a_m-1)t^m  + t^m \\
		& = & \sum_{i=0}^{m+1}a_it^i
	\end{eqnarray*} 
	So $f(t) = g(t)$, thus we have $\Delta_K (t)=\Delta_{K'}(t)$.
	
	Furthermore, by construction, we have that $r_1(K') \leq m=d-1$.
	By Proposition~\ref{prop:hgrn_bound} and the classical fact that $\deg(\Delta_{K'}) \leq 2g(K')$, we have
	$$r_1(K')\geq g(K')-1\geq \frac{1}{2}\deg(\Delta_{K'}(t))-1 = d-1,$$
	as desired.
\end{proof}

\newpage

\begin{table}[!ht]
    \centering
	\resizebox*{!}{.95\dimexpr\textheight-2\baselineskip\relax}{%
    \begin{tabular}{|l|l|l|l|}
    \hline
        Structure & Ribbon Code & Alexander Polynomial & Det \\ \hline
        1 & ([1,2,2,2,1,2]) & $1 - t + t^2 - 3t^3 + t^4 - t^5 + t^6$ & 9 \\ \hline
        1 & ([1,2,2,2,-1,2]) & $1 - t - t^3 + 3t^4 - t^5 - t^7 + t^8$ & 9 \\ \hline
        1 & ([1,2,2,1,1,2]) & $2 - 5t^2 + 2t^4$ & 1 \\ \hline
        1 & ([1,2,2,-1,-1,2]) & $1 - 2t^2 + 3t^4 - 2t^6 + t^8$ & 1 \\ \hline
        1 & ([1,2,1,2,1,2]) & $6 - 13t + 6t^2$ & 25 \\ \hline
        1 & ([1,2,1,2,-1,2]) & $2 - 6t + 9t^2 - 6t^3 + 2t^4$ & 25 \\ \hline
        1 & ([1,2,1,-2,1,2]) & $1 - 3t + 5t^2 - 7t^3 + 5t^4 - 3t^5 + t^6$ & 25 \\ \hline
        1 & ([1,2,1,-2,-1,2]) & $1 - 6t + 11t^2 - 6t^3 + t^4$ & 25 \\ \hline
        1 & ([1,2,-1,2,1,2]) & $1 - 3t + 5t^2 - 7t^3 + 5t^4 - 3t^5 + t^6$ & 25 \\ \hline
        1 & ([1,2,-1,2,-1,2]) & $1 - 2t + 3t^2 - 4t^3 + 5t^4 - 4t^5 + 3t^6 - 2t^7 + t^8$ & 25 \\ \hline
        2.2 & ([1,2,2,3,2],[2,1,3]) & $1 - 2t + 4t^3 - 7t^4 + 4t^5 - 2t^7 + t^8$ & 9 \\ \hline
        2.2 & ([1,2,2,-3,2],[2,1,3]) & $1 - t - 3t^2 + 7t^3 - 3t^4 - t^5 + t^6$ & 9 \\ \hline
        2.2 & ([1,-2,-2,3,2],[2,1,3]) & $2 - 5t + 2t^2$ & 9 \\ \hline
        2.2 & ([1,-2,-2,-3,2],[2,1,3]) & $1 - t - 3t^2 + 7t^3 - 3t^4 - t^5 + t^6$ & 9 \\ \hline
        2.3 & ([1,2,3,3,2],[2,1,3]) & $1 - 2t + t^2 + 2t^3 - 5t^4 + 2t^5 + t^6 - 2t^7 + t^8$ & 1 \\ \hline
        2.3 & ([1,2,-3,-3,2],[2,1,3]) & $1 - t - t^2 + 3t^3 - t^4 - t^5 + t^6$ & 1 \\ \hline
        2.3 & ([1,-2,3,3,2],[2,1,3]) & $1 - 3t - t^2 + 7t^3 - t^4 - 3t^5 + t^6$ & 1 \\ \hline
        2.3 & ([1,-2,-3,-3,2],[2,1,3]) & $2 - 3t - 2t^2 + 7t^3 - 2t^4 - 3t^5 + 2t^6$ & 1 \\ \hline %
        2.4 & ([1,2,1,3,2],[2,1,3]) & $1 - 3t^2 + t^4$ & 1 \\ \hline
        2.4 & ([1,2,1,-3,2],[2,1,3]) & $1$ & 1 \\ \hline
        2.4 & ([1,2,-1,3,2],[2,1,3]) & $1 - 3t + 2t^2 + 3t^3 - 7t^4 + 3t^5 + 2t^6 - 3t^7 + t^8$ & 1 \\ \hline
        2.4 & ([1,2,-1,-3,2],[2,1,3]) & $1 - t - t^2 + 3t^3 - t^4 - t^5 + t^6$ & 1 \\ \hline
        2.4 & ([1,-2,1,3,2],[2,1,3]) & $1$ & 1 \\ \hline
        2.4 & ([1,-2,1,-3,2],[2,1,3]) & $1 - t - t^2 + 3t^3 - t^4 - t^5 + t^6$ & 1 \\ \hline
        2.4 & ([1,-2,-1,3,2],[2,1,3]) & $1 - 3t - t^2 + 7t^3 - t^4 - 3t^5 + t^6$ & 1 \\ \hline
        2.4 & ([1,-2,-1,-3,2],[2,1,3]) & $1$ & 1 \\ \hline
        2.4 & ([1,2,3,1,2],[2,1,3]) & $1 - t + t^2 - 3t^3 + t^4 - t^5 + t^6$ & 9 \\ \hline
        2.4 & ([1,2,3,-1,2],[2,1,3]) & $1 - 2t + 3t^2 - 2t^3 + t^4$ & 9 \\ \hline
        2.4 & ([1,2,-3,1,2],[2,1,3]) & $1 - t - 3t^2 + 7t^3 - 3t^4 - t^5 + t^6$ & 9 \\ \hline
        2.4 & ([1,2,-3,-1,2],[2,1,3]) & $1 - 2t + 3t^2 - 2t^3 + t^4$ & 9 \\ \hline
        2.4 & ([1,-2,3,1,2],[2,1,3]) & $2 - 5t + 2t^2$ & 9 \\ \hline
        2.4 & ([1,-2,3,-1,2],[2,1,3]) & $2 - 5t + 2t^2$ & 9 \\ \hline
        2.4 & ([1,-2,-3,1,2],[2,1,3]) & $2 - 2t - 4t^2 + 9t^3 - 4t^4 - 2t^5 + 2t^6$ & 9 \\ \hline
        2.4 & ([1,-2,-3,-1,2],[2,1,3]) & $2 - 5t + 2t^2$ & 9 \\ \hline
	\end{tabular}}
    \caption{Ribbon codes and Alexander polynomials for Structures 1 and 2}
    \label{table:rc1-2}
\end{table}

\begin{table}[!ht]
    \centering
	\resizebox*{!}{.95\dimexpr\textheight-2\baselineskip\relax}{%
    \begin{tabular}{|l|l|l|l|}
    \hline
        Structure & Ribbon Code & Alexander Polynomial & Det \\ \hline
        3.1 & ([1,2,3,2],[2,3,1,3]) & $2 - 6t + 10t^2 - 13t^3 + 10t^4 - 6t^5 + 2t^6$ & 49 \\ \hline
        3.1 & ([1,2,3,2],[2,3,-1,3]) & $1 - 3t + 6t^2 - 9t^3 + 11t^4 - 9t^5 + 6t^6 - 3t^7 + t^8$ & 49 \\ \hline
        3.1 & ([1,2,3,2],[2,-3,1,3]) & $1 - 5t + 11t^2 - 15t^3 + 11t^4 - 5t^5 + t^6$ & 49 \\ \hline
        3.1 & ([1,2,3,2],[2,-3,-1,3]) & $1 - 5t + 11t^2 - 15t^3 + 11t^4 - 5t^5 + t^6$ & 49 \\ \hline
        3.1 & ([1,2,-3,2],[2,3,1,3]) & $3 - 12t + 19t^2 - 12t^3 + 3t^4$ & 49 \\ \hline
        3.1 & ([1,2,-3,2],[2,3,-1,3]) & $1 - 5t + 11t^2 - 15t^3 + 11t^4 - 5t^5 + t^6$ & 49 \\ \hline
        3.1 & ([1,2-,3,2],[2,-3,1,3]) & $4 - 12t + 17t^2 - 12t^3 + 4t^4$ & 49 \\ \hline
        3.1 & ([1,2,-3,2],[2,-3,-1,3]) & $2 - 12t + 21t^2 - 12t^3 + 2t^4$ & 49 \\ \hline
        
		3.1 & ([1,2,1,2],[2,3,1,3]) & $4 - 20t + 33t^2 - 20t^3 + 4t^4$ & 81 \\ \hline
		3.1 & ([1,2,1,2],[2,3,-1,3]) & $2 - 9t + 18t^2 - 23t^3 + 18t^4 - 9t^5 + 2t^6$ & 81 \\ \hline
		3.1 & ([1,2,1,2],[2,-3,1,3]) & $2 - 9t + 18t^2 - 23t^3 + 18t^4 - 9t^5 + 2t^6$ & 81 \\ \hline
		3.1 & ([1,2,1,2],[2,-3,-1,3])	 & $4 - 20t + 33t^2 - 20t^3 + 4t^4$ & 81 \\ \hline
		3.1 & ([1,2,-1,2],[2,3,1,3]) & $2 - 9t + 18t^2 - 23t^3 + 18t^4 - 9t^5 + 2t^6$ & 81 \\ \hline
		3.1 & ([1,2,-1,2],[2,3,-1,3])	 & $1 - 4t + 10t^2 - 16t^3 + 19t^4 - 16t^5 + 10t^6 - 4t^7 + t^8$ & 81 \\ \hline
		3.1 & ([1,2,-1,2],[2,-3,1,3])	 & $1 - 4t + 10t^2 - 16t^3 + 19t^4 - 16t^5 + 10t^6 - 4t^7 + t^8$ & 81 \\ \hline
        3.1 & ([1,2,-1,2],[2,-3,-1,3]) & $2 - 9t + 18t^2 - 23t^3 + 18t^4 - 9t^5 + 2t^6$ & 81 \\ \hline

        3.1 & ([1,2,-3,2],[2,-1,-1,3]) & $1 - 3t - t^2 + 7t^3 - t^4 - 3t^5 + t^6$ & 1 \\ \hline
        3.1 & ([1,2,-3,2],[2,1,1,3]) & $2 - 3t - 2t^2 + 7t^3 - 2t^4 - 3t^5 + 2t^6$ & 1 \\ \hline
        3.1 & ([1,2,3,2],[2,-1,-1,3])	 & $1 - t - t^2 + 3t^3 - t^4 - t^5 + t^6$ & 1 \\ \hline
        3.1 & ([1,2,3,2],[2,1,1,3]) & $1 - 2t + t^2 + 2t^3 - 5t^4 + 2t^5 + t^6 - 2t^7 + t^8$ & 1 \\ \hline


        4.1 & ([1,2,3,2],[2,4,3],[3,1,4]) & $3 - 13t + 27t^2 - 35t^3 + 27t^4 - 13t^5 + 3t^6$ & 121 \\ \hline
        4.1 & ([1,2,3,2],[2,-4,3],[3,1,4]) & $1 - 5t + 14t^2 - 25t^3 + 31t^4 - 25t^5 + 14t^6 - 5t^7 + t^8$ & 121 \\ \hline
        4.1 & ([1,2,-3,2],[2,4,3],[3,1,4]) & $1 - 9t + 29t^2 - 43t^3 + 29t^4 - 9t^5 + t^6$ & 121 \\ \hline
        4.1 & ([1,2,-3,2],[2,-4,3],[3,1,4]) & $2 - 12t + 28t^2 - 37t^3 + 28t^4 - 12t^5 + 2t^6$ & 121 \\ \hline
        4.1 & ([1,-2,3,2],[2,4,3],[3,1,4]) & $2 - 11t + 28t^2 - 39t^3 + 28t^4 - 11t^5 + 2t^6$ & 121 \\ \hline
        4.1 & ([1,-2,3,2],[2,-4,3],[3,1,4]) & $2 - 12t + 28t^2 - 37t^3 + 28t^4 - 12t^5 + 2t^6$ & 121 \\ \hline
        4.1 & ([1,-2,-3,2],[2,4,3],[3,1,4]) & $1 - 6t + 15t^2 - 24t^3 + 29t^4 - 24t^5 + 15t^6 - 6t^7 + t^8$ & 121 \\ \hline
        4.1 & ([1,-2,-3,2],[2,-4,3],[3,1,4]) & $1 - 5t + 14t^2 - 25t^3 + 31t^4 - 25t^5 + 14t^6 - 5t^7 + t^8$ & 121 \\ \hline
        4.2 & ([1,4,3,2],[2,1,3],[3,2,4]) & $2 - 8t + 18t^2 - 25t^3 + 18t^4 - 8t^5 + 2t^6$ & 81 \\ \hline
        4.2 & ([1,4,3,2],[2,-1,3],[3,2,4]) & $6 - 20t + 29t^2 - 20t^3 + 6t^4$ & 81 \\ \hline
        4.2 & ([1,4,-3,2],[2,1,3],[3,2,4]) & $2 - 9t + 18t^2 - 23t^3 + 18t^4 - 9t^5 + 2t^6$ & 81 \\ \hline
        4.2 & ([1,4,-3,2],[2,-1,3],[3,2,4]) & $1 - 7t + 19t^2 - 27t^3 + 19t^4 - 7t^5 + t^6$ & 81 \\ \hline
        4.2 & ([1,-4,3,2],[2,1,3],[3,2,4]) & $1 - 7t + 19t^2 - 27t^3 + 19t^4 - 7t^5 + t^6$ & 81 \\ \hline
        4.2 & ([1,-4,3,2],[2,-1,3],[3,2,4]) & $1 - 7t + 19t^2 - 27t^3 + 19t^4 - 7t^5 + t^6$ & 81 \\ \hline
        4.2 & ([1,-4,-3,2],[2,1,3],[3,2,4]) & $1 - 4t + 10t^2 - 16t^3 + 19t^4 - 16t^5 + 10t^6 - 4t^7 + t^8$ & 81 \\ \hline
        4.2 & ([1,-4,-3,2],[2,-1,3],[3,2,4]) & $1 - 5t + 10t^2 - 15t^3 + 19t^4 - 15t^5 + 10t^6 - 5t^7 + t^8$ & 81 \\ \hline
        
        5.2 & ([1,3,2],[2,1,4,3],[3,2,4]) & $2 - 12t + 21t^2 - 12t^3 + 2t^4$ & 49 \\ \hline
        5.2 & ([1,3,2],[2,1,-4,3],[3,2,4]) & $1 - 5t + 11t^2 - 15t^3 + 11t^4 - 5t^5 + t^6$ & 49 \\ \hline
        5.2 & ([1,3,2],[2,-1,4,3],[3,2,4]) & $1 - 5t + 11t^2 - 15t^3 + 11t^4 - 5t^5 + t^6$ & 49 \\ \hline
        5.2 & ([1,3,2],[2,-1,-4,3],[3,2,4]) & $4 - 12t + 17t^2 - 12t^3 + 4t^4$ & 49 \\ \hline
        5.2 & ([1,-3,2],[2,1,4,3],[3,2,4]) & $1 - 5t + 11t^2 - 15t^3 + 11t^4 - 5t^5 + t^6$ & 49 \\ \hline
        5.2 & ([1,-3,2],[2,1,-4,3],[3,2,4]) & $4 - 12t + 17t^2 - 12t^3 + 4t^4$ & 49 \\ \hline
        5.2 & ([1,-3,2],[2,-1,4,3],[3,2,4]) & $1 - 4t + 6t^2 - 8t^3 + 11t^4 - 8t^5 + 6t^6 - 4t^7 + t^8$ & 49 \\ \hline
        5.2 & ([1,-3,2],[2,-1,-4,3],[3,2,4]) & $1 - 5t + 11t^2 - 15t^3 + 11t^4 - 5t^5 + t^6$ & 49 \\ \hline
	\end{tabular}}
    \caption{Ribbon codes and Alexander polynomials for Structures 3, 4, and~5}
    \label{table:rc3-5}
\end{table}

\begin{table}[!ht]
    \centering
	\resizebox*{!}{.95\dimexpr\textheight-2\baselineskip\relax}{%
    \begin{tabular}{|l|l|l|l|}
    \hline
        Structure & Ribbon Code & Alexander Polynomial & Det \\ \hline
        6.1.2 & ([1,4,2,4],[4,3,2],[4,1,3]) & $1 - 4t + 9t^2 - 16t^3 + 21t^4 - 16t^5 + 9t^6 - 4t^7 + t^8$ & 81 \\ \hline
        6.1.2 & ([1,4,2,4],[4,3,2],[4,-1,3]) & $2 - 9t + 18t^2 - 23t^3 + 18t^4 - 9t^5 + 2t^6$ & 81 \\ \hline
        6.1.2 & ([1,4,2,4],[4,-3,2],[4,1,3]) & $2 - 9t + 18t^2 - 23t^3 + 18t^4 - 9t^5 + 2t^6$ & 81 \\ \hline
        6.1.2 & ([1,4,2,4],[4,-3,2],[4,-1,3]) & $1 - 7t + 19t^2 - 27t^3 + 19t^4 - 7t^5 + t^6$ & 81 \\ \hline
        6.1.2 & ([1,4,-2,4],[4,3,2],[4,1,3]) & $2 - 9t + 18t^2 - 23t^3 + 18t^4 - 9t^5 + 2t^6$ & 81 \\ \hline
        6.1.2 & ([1,4,-2,4],[4,3,2],[4,-1,3]) & $1 - 7t + 19t^2 - 27t^3 + 19t^4 - 7t^5 + t^6$ & 81 \\ \hline
        6.1.2 & ([1,4,-2,4],[4,-3,2],[4,1,3]) & $1 - 7t + 19t^2 - 27t^3 + 19t^4 - 7t^5 + t^6$ & 81 \\ \hline
        6.1.2 & ([1,4,-2,4],[4,-3,2],[4,-1,3]) & $1 - 7t + 19t^2 - 27t^3 + 19t^4 - 7t^5 + t^6$ & 81 \\ \hline
        6.2.1 & ([1,3,1,4],[4,1,2],[4,2,3]) & $1 - t - 7t^2 + 15t^3 - 7t^4 - t^5 + t^6$ & 25 \\ \hline
        6.2.1 & ([1,3,1,4],[4,1,2],[4,-2,3]) & $1 - 3t + 5t^2 - 7t^3 + 5t^4 - 3t^5 + t^6$ & 25 \\ \hline
        6.2.1 & ([1,3,-1,4],[4,1,2],[4,2,3]) & $2 - 6t + 9t^2 - 6t^3 + 2t^4$ & 25 \\ \hline
        6.2.1 & ([1,3,-1,4],[4,1,2],[4,-2,3]) & $2 - 6t + 9t^2 - 6t^3 + 2t^4$ & 25 \\ \hline
        6.2.1 & ([1,-3,1,4],[4,1,2],[4,2,3]) & $1 - 3t + 9t^3 - 15t^4 + 9t^5 - 3t^7 + t^8$ & 25 \\ \hline
        6.2.1 & ([1,-3,1,4],[4,1,2],[4,-2,3]) & $2 - 6t + 9t^2 - 6t^3 + 2t^4$ & 25 \\ \hline
        6.2.1 & ([1,-3,-1,4],[4,1,2],[4,2,3]) & $1 - 6t + 11t^2 - 6t^3 + t^4$ & 25 \\ \hline
        6.2.1 & ([1,-3,-1,4],[4,1,2],[4,-2,3]) & $1 - 3t + 5t^2 - 7t^3 + 5t^4 - 3t^5 + t^6$ & 25 \\ \hline
        
        6.2.2 & ([1,2,2,4],[4,3,2],[4,1,3]) & $2 - 4t - 2t^2 + 9t^3 - 2t^4 - 4t^5 + 2t^6$ & 1 \\ \hline
        6.2.2 & ([1,2,2,4],[4,3,2],[4,-1,3]) & $1 - 3t + 2t^2 + 3t^3 - 7t^4 + 3t^5 + 2t^6 - 3t^7 + t^8$ & 1 \\ \hline
        6.2.2 & ([1,2,2,4],[4,-3,2],[4,1,3]) & $1 - 3t + 2t^2 + 3t^3 - 7t^4 + 3t^5 + 2t^6 - 3t^7 + t^8$ & 1 \\ \hline
        6.2.2 & ([1,2,2,4],[4,-3,2],[4,-1,3]) & $1 - 4t + 4t^2 + 4t^3 - 11t^4 + 4t^5 + 4t^6 - 4t^7 + t^8$ & 1 \\ \hline
        6.2.2 & ([1,3,2,4],[4,1,2],[4,2,3]) & $2 - 5t + 2t^2$ & 9 \\ \hline
        6.2.2 & ([1,3,2,4],[4,-1,2],[4,2,3]) & $1 - 2t + 3t^2 - 2t^3 + t^4$ & 9 \\ \hline
        6.2.2 & ([1,3,-2,4],[4,1,2],[4,2,3]) & $1 - t - 3t^2 + 7t^3 - 3t^4 - t^5 + t^6$ & 9 \\ \hline
        6.2.2 & ([1,3,-2,4],[4,-1,2],[4,2,3]) & $2 - 5t + 2t^2$ & 9 \\ \hline
        6.2.2 & ([1,-3,2,4],[4,1,2],[4,2,3]) & $2 - 5t + 2t^2$ & 9 \\ \hline
        6.2.2 & ([1,-3,2,4],[4,-1,2],[4,2,3]) & $1 - 2t + 3t^2 - 2t^3 + t^4$ & 9 \\ \hline
        6.2.2 & ([1,-3,-2,4],[4,1,2],[4,2,3]) & $1 - 4t + 3t^2 + 6t^3 - 13t^4 + 6t^5 + 3t^6 - 4t^7 + t^8$ & 9 \\ \hline
        6.2.2 & ([1,-3,-2,4],[4,-1,2],[4,2,3]) & $1 - 3t + 2t^2 + 5t^3 - 11t^4 + 5t^5 + 2t^6 - 3t^7 + t^8$ & 9 \\ \hline
        7.2 & ([1,4,2],[2,1,5],[5,2,3],[5,3,4]) & $1 - 6t + 18t^2 - 36t^3 + 47t^4 - 36t^5 + 18t^6 - 6t^7 + t^8$ & 169 \\ \hline
        7.2 & ([1,4,2],[2,1,5],[5,2,3],[5,-3,4]) & $1 - 7t + 20t^2 - 35t^3 + 43t^4 - 35t^5 + 20t^6 - 7t^7 + t^8$ & 169 \\ \hline
        7.2 & ([1,4,2],[2,1,5],[5,-2,3],[5,3,4]) & $1 - 7t + 20t^2 - 35t^3 + 43t^4 - 35t^5 + 20t^6 - 7t^7 + t^8$ & 169 \\ \hline
        7.2 & ([1,4,2],[2,1,5],[5,-2,3],[5,-3,4]) & $2 - 14t + 40t^2 - 57t^3 + 40t^4 - 14t^5 + 2t^6$ & 169 \\ \hline
        7.2 & ([1,4,2],[2,-1,5],[5,2,3],[5,3,4]) & $3 - 17t + 39t^2 - 51t^3 + 39t^4 - 17t^5 + 3t^6$ & 169 \\ \hline
        7.2 & ([1,4,2],[2,-1,5],[5,2,3],[5,-3,4]) & $1 - 6t + 19t^2 - 36t^3 + 45t^4 - 36t^5 + 19t^6 - 6t^7 + t^8$ & 169 \\ \hline
        7.2 & ([1,4,2],[2,-1,5],[5,-2,3],[5,3,4]) & $1 - 6t + 19t^2 - 36t^3 + 45t^4 - 36t^5 + 19t^6 - 6t^7 + t^8$ & 169 \\ \hline
        7.2 & ([1,4,2],[2,-1,5],[5,-2,3],[5,-3,4]) & $3 - 17t + 39t^2 - 51t^3 + 39t^4 - 17t^5 + 3t^6$ & 169 \\ \hline

        8.2 & ([1,2,5],[2,3,5],[5,4,3],[5,1,4]) & $4 - 22t + 52t^2 - 69t^3 + 52t^4 - 22t^5 + 4t^6$ & 225 \\ \hline
        8.2 & ([1,2,5],[2,3,5],[5,4,3],[5,-1,4]) & $1 - 7t + 24t^2 - 49t^3 + 63t^4 - 49t^5 + 24t^6 - 7t^7 + t^8$ & 225 \\ \hline
        8.2 & ([1,2,5],[2,3,5],[5,-4,3],[5,-1,4]) & $1 - 8t + 26t^2 - 48t^3 + 59t^4 - 48t^5 + 26t^6 - 8t^7 + t^8$ & 225 \\ \hline
        8.2 & ([1,2,5],[2,-3,5],[5,4,3],[5,-1,4]) & $1 - 8t + 26t^2 - 48t^3 + 59t^4 - 48t^5 + 26t^6 - 8t^7 + t^8$ & 225 \\ \hline
	\end{tabular}}
    \caption{Ribbon codes and Alexander polynomials for Structures 6 and 7}
    \label{table:rc6-8}
\end{table}

%

\begin{figure}[h!]
\centering
\begin{subfigure}{.33\textwidth}
  \centering
  \includegraphics[width=.9\linewidth]{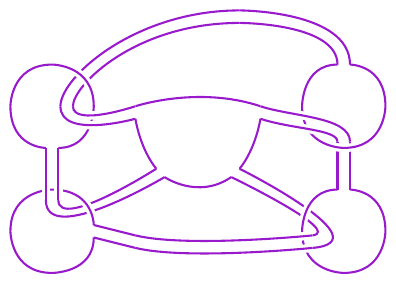}
  \caption{$12a_{427}$}
\end{subfigure}%
\begin{subfigure}{.33\textwidth}
  \centering
  \includegraphics[width=.9\linewidth]{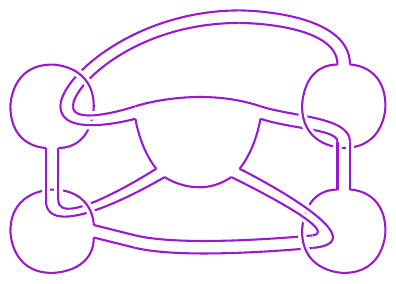}
  \caption{$12a_{435}$}
\end{subfigure}%
\begin{subfigure}{.33\textwidth}
  \centering
  \includegraphics[width=.9\linewidth]{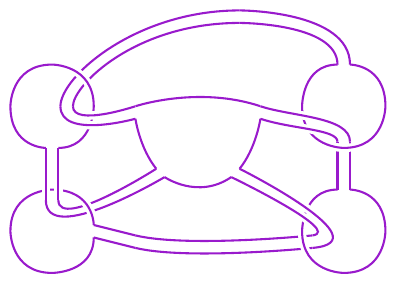}
  \caption{$12a_{464}$}
\end{subfigure} \\
\begin{subfigure}{.33\textwidth}
  \centering
  \includegraphics[width=.9\linewidth]{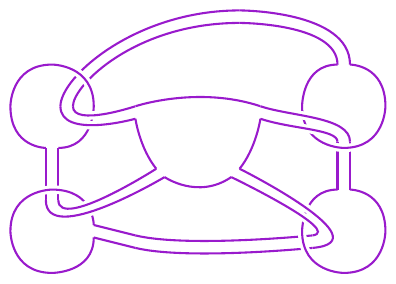}
  \caption{$12a_{975}$}
\end{subfigure}%
\begin{subfigure}{.33\textwidth}
  \centering
  \includegraphics[width=.9\linewidth]{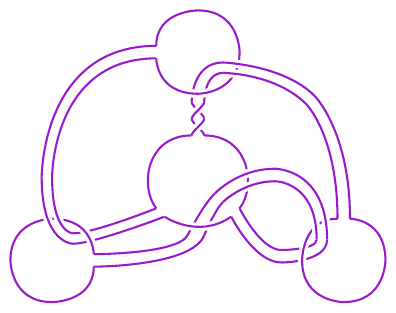}
  \caption{$12n_4$}
\end{subfigure}%
\begin{subfigure}{.33\textwidth}
  \centering
  \includegraphics[width=.9\linewidth]{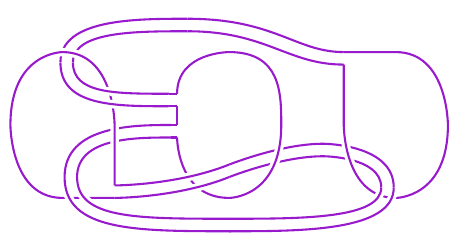}
  \caption{$12n_{23}$}
\end{subfigure} \\
\begin{subfigure}{.33\textwidth}
  \centering
  \includegraphics[width=.9\linewidth]{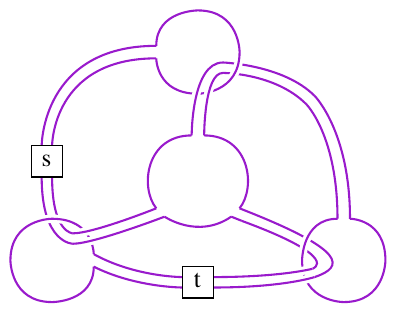}
  \caption{\phantom{yo}}
\end{subfigure}%
\begin{subfigure}{.33\textwidth}
  \centering
  \includegraphics[width=.9\linewidth]{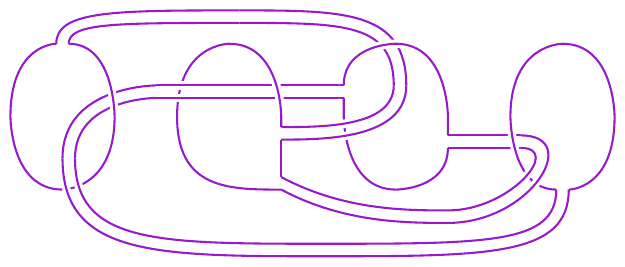}
  \caption{$12n_{43}$}
\end{subfigure}%
\begin{subfigure}{.33\textwidth}
  \centering
  \includegraphics[width=.9\linewidth]{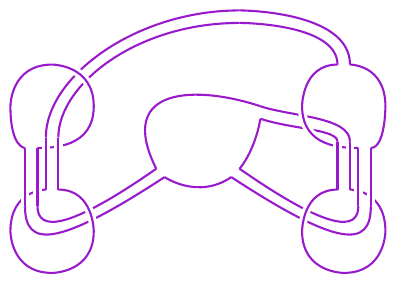}
  \caption{$12n_{49}$}
\end{subfigure} \\
\begin{subfigure}{.33\textwidth}
  \centering
  \includegraphics[width=.9\linewidth]{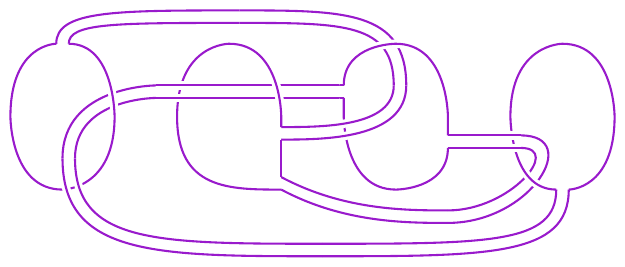}
  \caption{$12n_{106}$}
\end{subfigure}%
\begin{subfigure}{.33\textwidth}
  \centering
  \includegraphics[width=.9\linewidth]{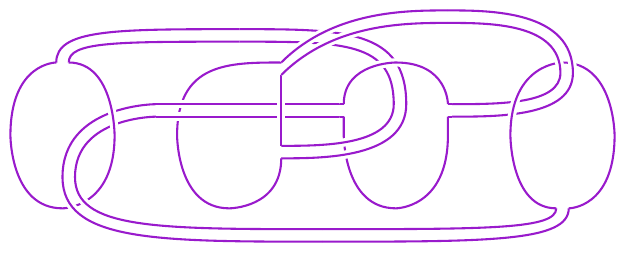}
  \caption{$12n_{170}$}
\end{subfigure}%
\begin{subfigure}{.33\textwidth}
  \centering
  \includegraphics[width=.9\linewidth]{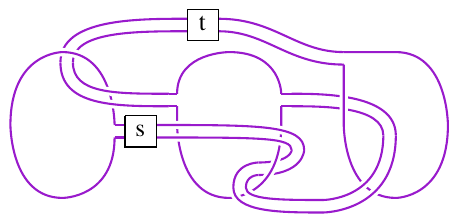}
  \caption{\phantom{yo}}
\end{subfigure}
\caption{(G) $12n_{24}$ ($s=-1,t=2$) and $12n_{312}$ ($s=0,t=2$), (L) $12n_{257}$ $(s=-1,t=1)$, $12n_{279}$ $(s=1,t=0)$, $12n_{394}$ $(s=-1,t=0)$, and $12n_{870}$~$(s=0,t=1)$}
\label{fig:ribbonA}
\end{figure}

\begin{figure}[h!]
\centering
\begin{subfigure}{.33\textwidth}
  \centering
  \includegraphics[width=.9\linewidth]{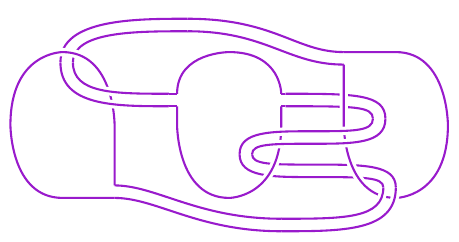}
  \caption{$12n_{288}$}
\end{subfigure}%
\begin{subfigure}{.33\textwidth}
  \centering
  \includegraphics[width=.9\linewidth]{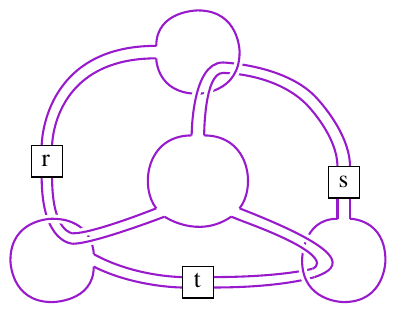}
  \caption{\phantom{yo}}
\end{subfigure}%
\begin{subfigure}{.33\textwidth}
  \centering
  \includegraphics[width=.9\linewidth]{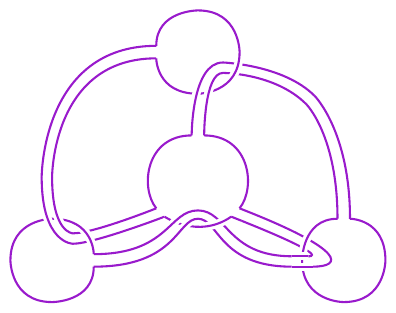}
  \caption{$12n_{380}$}
\end{subfigure} \\
\begin{subfigure}{.33\textwidth}
  \centering
  \includegraphics[width=.9\linewidth]{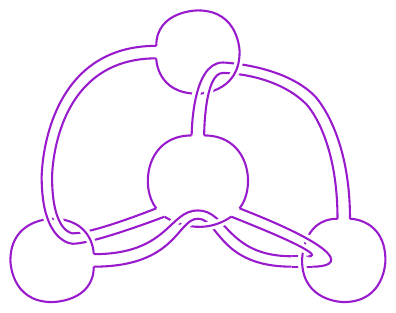}
  \caption{$12n_{399}$}
\end{subfigure}
\begin{subfigure}{.33\textwidth}
  \centering
  \includegraphics[width=.9\linewidth]{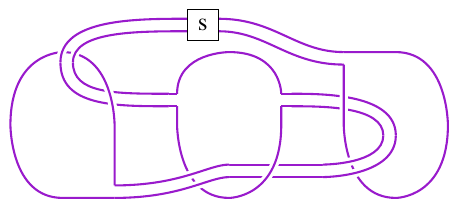}
  \caption{\phantom{yo}}
\end{subfigure}%
\begin{subfigure}{.33\textwidth}
  \centering
  \includegraphics[width=.9\linewidth]{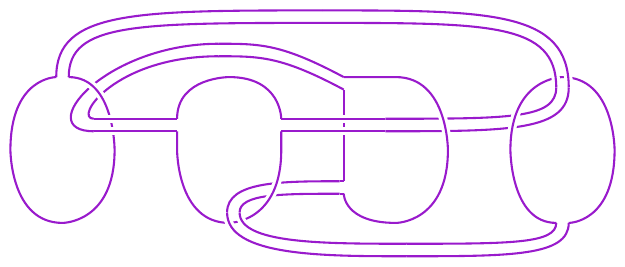}
  \caption{$12n_{420}$}
\end{subfigure} \\
\begin{subfigure}{.33\textwidth}
  \centering
  \includegraphics[width=.9\linewidth]{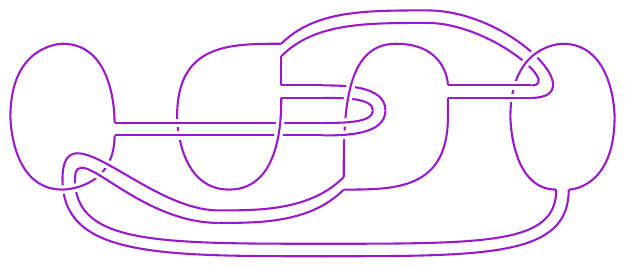}
  \caption{$12n_{504}$}
\end{subfigure}%
\begin{subfigure}{.33\textwidth}
  \centering
  \includegraphics[width=.9\linewidth]{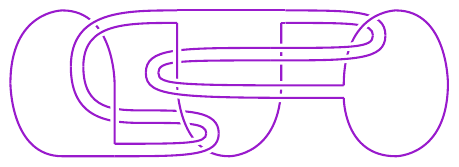}
  \caption{$12n_{553}$}
\end{subfigure}%
\begin{subfigure}{.33\textwidth}
  \centering
  \includegraphics[width=.9\linewidth]{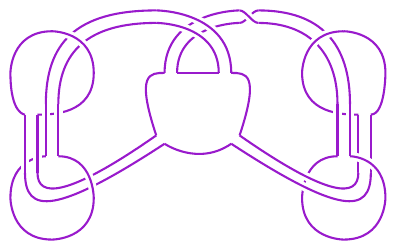}
  \caption{$12n_{556}$}
\end{subfigure} \\
\begin{subfigure}{.33\textwidth}
  \centering
  \includegraphics[width=.9\linewidth]{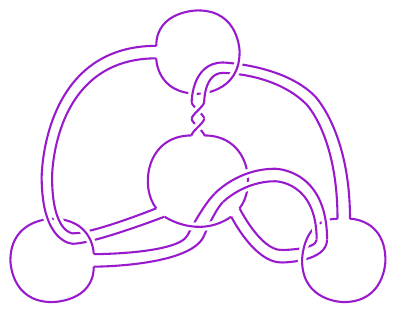}
  \caption{$12n_{636}$}
\end{subfigure}%
\begin{subfigure}{.33\textwidth}
  \centering
  \includegraphics[width=.9\linewidth]{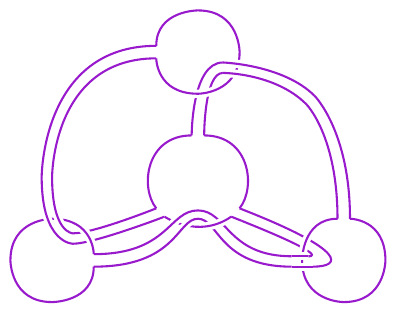}
  \caption{$12n_{657}$}
\end{subfigure}%
\begin{subfigure}{.33\textwidth}
  \centering
  \includegraphics[width=.9\linewidth]{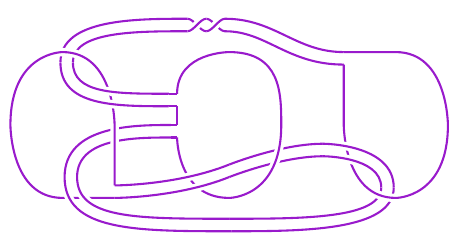}
  \caption{$12n_{676}$}
\end{subfigure} \\
\begin{subfigure}{.33\textwidth}
  \centering
  \includegraphics[width=.9\linewidth]{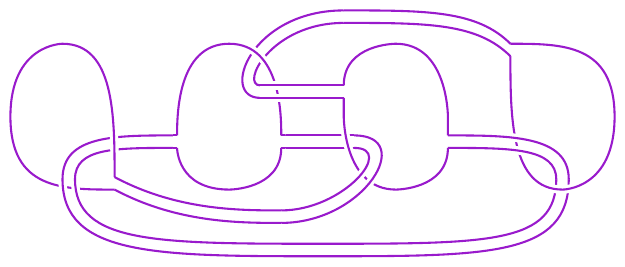}
  \caption{$12n_{706}$}
\end{subfigure}%
\begin{subfigure}{.33\textwidth}
  \centering
  \includegraphics[width=.9\linewidth]{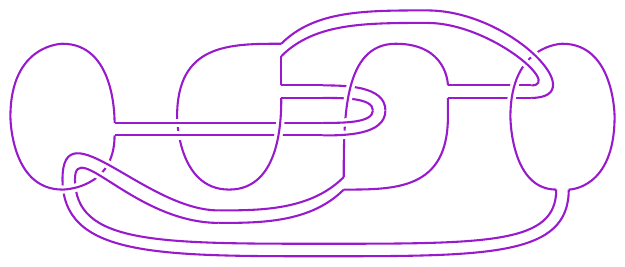}
  \caption{$12n_{802}$}
\end{subfigure}%
\begin{subfigure}{.33\textwidth}
  \centering
  \includegraphics[width=.9\linewidth]{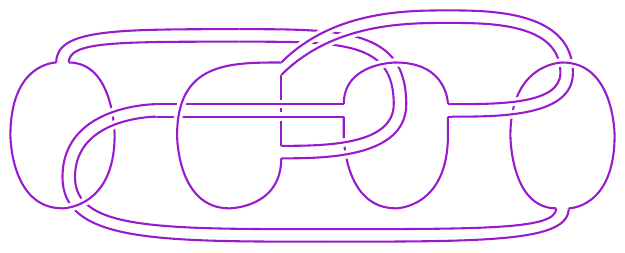}
  \caption{$12n_{876}$}
\end{subfigure}
\caption{(B) $12n_{360}$ ($r=1,s=1,t=-1$), $12n_{393}$ ($r=2,s=0,t=0$), and $12n_{397}$ ($r=2,s=-1,t=0$), (E) $12n_{414}$ ($s=-3)$ and $12n_{768}$ $(s=4)$}
\label{fig:ribbonB}
\end{figure}

\newpage

\bibliographystyle{amsalpha}
\bibliography{ribbon_numbers.bib}

\end{document}